\documentclass[onefignum,onetabnum]{siamart190516}

\usepackage{lipsum}
\usepackage{amsfonts}
\usepackage{graphicx}
\usepackage{epstopdf}
\usepackage{algpseudocode}
\usepackage{multirow}
\usepackage{adjustbox}
\usepackage{booktabs}
\usepackage{cprotect}

\ifpdf
  \DeclareGraphicsExtensions{.eps,.pdf,.png,.jpg}
\else
  \DeclareGraphicsExtensions{.eps}
\fi

\newsiamremark{remark}{Remark}
\newsiamremark{hypothesis}{Hypothesis}
\crefname{hypothesis}{Hypothesis}{Hypotheses}
\newsiamthm{claim}{Claim}

\headers{Two-level Nystr\"om--Schur Preconditioner for SPD Matrices}{H. Al Daas, T. Rees, and J. Scott}

\title{Two-level Nystr\"om--Schur Preconditioner for Sparse Symmetric Positive Definite Matrices\thanks{Submitted to the editors January 28, 2021.}}

\author{Hussam Al Daas\thanks{STFC Rutherford Appleton Laboratory, Harwell Campus, Didcot, Oxfordshire, OX11 0QX, UK 
  (\email{hussam.al-daas@stfc.ac.uk},
  \email{tyrone.rees@stfc.ac.uk}, 
  \email{jennifer.scott@stfc.ac.uk}).}
\and Tyrone Rees\footnotemark[2]
\and Jennifer Scott\footnotemark[2] \thanks{School of Mathematical, Physical and Computational Sciences,
University of Reading, Reading RG6 6AQ, UK}
}

\usepackage{amsopn}

\newcommand{\R}{\mathbb{R}}
\newcommand{\calP}{\mathcal{P}}
\newcommand{\calPD}{\mathcal{P}_{\text{\tiny{DEF}}}}
\newcommand{\calPAD}{\mathcal{P}_{\text{\tiny{A-DEF}}}}

\newcommand{\calPADi}[1]{\mathcal{P}_{\text{\tiny{#1A-DEF}}}}
\newcommand{\calM}{\mathcal{M}}
\newcommand{\calMADi}[1]{\mathcal{M}_{\text{\tiny{#1A-DEF}}}}

\ifpdf
\hypersetup{
  pdftitle={Two-level Nystr\"om--Schur Preconditioner for Sparse Symmetric Positive Definite Matrices},
  pdfauthor={H. Al Daas, T. Rees, and J.~A. Scott}
}
\fi

\newcommand{\cblue}[1]{{{#1}}}

\begin{document}

\maketitle

\begin{abstract}
  Randomized methods are becoming increasingly popular in numerical linear algebra.
  However, few attempts have been made to use them in developing preconditioners.
  Our interest lies in solving large-scale sparse symmetric positive definite linear systems of equations
  where the system matrix is preordered to doubly bordered block diagonal form (for example, using a
  nested dissection ordering).
  We investigate the use of randomized methods to construct high quality preconditioners.
  In particular, we propose a new and efficient approach that employs Nystr\"om's method for computing low rank approximations
  to develop robust algebraic two-level preconditioners. Construction of the new 
  preconditioners involves iteratively solving a smaller but denser symmetric positive definite
  Schur complement system with multiple right-hand sides. Numerical experiments 
  on problems coming from a range of application areas demonstrate that this inner system
  can be solved cheaply using block conjugate gradients and that using a large convergence tolerance 
  to limit the cost does
  not adversely affect the quality of the resulting Nystr\"om--Schur two-level preconditioner.
\end{abstract}



\section{Introduction}
Large scale linear systems of equations arise in a wide range of real-life applications.
Since the 1970s, sparse direct methods, such as LU, Cholesky, and LDLT factorizations, have been studied in depth and 
library quality software is available (see, for example, \cite{DufER17} and the references therein).
However, their memory requirements and the difficulties in developing effective
parallel implementations can limit their scope for solving extremely large problems,
unless they are used in combination with an iterative approach.
Iterative methods are attractive because they have low memory requirements and are
simpler to parallelize.
In this work, our interest is in using the conjugate gradient (CG) method
to solve large sparse symmetric positive definite (SPD) systems of the form
\begin{equation}
  \label{eq:linear_system}
  Ax = b,
\end{equation}
where $A \in \R^{n\times n}$ is SPD, $b\in\R^n$ is the given right-hand side, and $x$ is the required solution.  
\cblue{The solution of SPD systems is ubiquitous in scientific
computing, being required in applications as diverse as
least-squares problems, non-linear optimization subproblems, Monte-Carlo simulations, 
finite element analysis, and Kalman filtering.
In the following, we assume no additional structure beyond a sparse SPD system.
}

It is well known that the approximate solution $x_k$ at iteration $k$ of the CG method satisfies
\begin{equation}
  \label{eq:k-th_error_cg}
  \|x_\star - x_k\|_A \le 2 \|x_\star - x_0\|_A \left(\frac{\sqrt{\kappa} - 1}{\sqrt{\kappa} + 1}\right)^k,
\end{equation}
where $x_\star$ is the exact solution, $x_0$ is the initial guess, $\|\cdot\|_A$ is the $A$-norm, 
and $\kappa(A) = \lambda_{\text{max}}/\lambda_{\text{min}}$ is the spectral condition number
($\lambda_{\text{max}}$ and $\lambda_{\text{min}}$ denote the largest and smallest eigenvalues of $A$).
The rate of convergence also depends on the distribution
of the eigenvalues  (as well as on $b$ and $x_0$): eigenvalues clustered
away from the origin lead to rapid convergence.
If $\kappa(A)$ is large and the eigenvalues of $A$ are evenly distributed,
the system needs to be preconditioned to enhance convergence. This can be done by
applying a linear operator $\calP$ to \cref{eq:linear_system}, where $\calP \in \R^{n\times n}$ is chosen 
so that the spectral condition number of $\calP A$ is small
and applying $\calP$ is inexpensive.
In some applications, knowledge of the provenance of $A$ can help in building an efficient preconditioner.
Algebraic preconditioners do not assume such knowledge, and include
incomplete Cholesky factorizations,
block Jacobi, Gauss--Seidel, 
and additive Schwarz; see, for example,  \cite{Saa03}.
These are referred to as {\em one-level} or {\em traditional} preconditioners \cite{DolJN15,TanMNV10}.
In general, algebraic preconditioners bound the largest eigenvalues of $\calP A$ but 
encounter difficulties in controlling the smallest eigenvalues, which can lie close to the origin,
hindering convergence.

Deflation strategies have been proposed to overcome the issues related to small eigenvalues.
As explained in \cite{KahR17}, the basic idea behind deflation is to ``hide'' certain
parts of the spectrum of the matrix from the CG method, such that the CG iteration
``sees'' a system that has a much smaller condition number than the original matrix. The part of the spectrum that
is hidden from CG is determined by the deflation subspace and the improvement in the
convergence rate of the deflated CG method is dependent on the choice of this subspace.
In the ideal case, the deflation subspace is the invariant subspace spanned by the eigenvectors associated 
with the smallest eigenvalues of $A$ and the convergence rate
is then governed by the ``effective'' spectral condition number associated with 
the remaining eigenvalues (that is, the ratio of the largest eigenvalue to the smallest remaining eigenvalue). 
The idea was first introduced in the late 1980s~\cite{Dos88,Nic87}, and has been
discussed and used by a number of researchers
\cite{AldG19,AldGJT21,FraV01,GriNY14,JonVVK09,JonVVS13,LiXS16,NatXDS11,SpiDHNPS14,SpiR13,VuiSM99,VuiSMW01}.
However, in most of these references, the deflation subspaces rely on the underlying partial differential equation and its discretization,
and cannot be applied to more general systems or used as ``black box'' preconditioners.
Algebraic two-level preconditioners 
have been proposed in \cite{AldJS21,GauGLN13,Gut14,NabV08,TanMNV10,TanNVE09}.
Recently, a two-level Schur complement preconditioner based on the power series approximation
was proposed in~\cite{ZheXS21}.

In recent years, the study of randomized methods 
has become an active and promising research area in the field of numerical linear algebra
(see, for example, \cite{HalMT11,Nak20} and the references therein).
The use of randomized methods to build preconditioners has been proposed in
a number of papers, including \cite{GriNY14,HigM19}.
The approach in \cite{GriNY14} starts by reordering the system matrix $A$ to a $2 \times 2$ doubly bordered
block diagonal form, which can be achieved using a nested dissection ordering.
The Schur complement system must then be solved.
Starting from a first-level preconditioner $\calP$, a deflation subspace 
is constructed via a low rank approximation.
Although deflation can be seen as a low rank correction, 
using randomized methods to estimate the low rank term is not straightforward 
because the deflation subspace is more likely to be associated 
with the invariant subspace corresponding to the smallest eigenvalues of the preconditioned matrix,
and not to its dominant subspace. 
\cblue{In~\cref{sec:background}, we review the ingredients involved in building our two-level preconditioner. This
  includes Nystr\"om's method for computing a
  low rank approximation of a matrix \cite{GitM16,HalMT11,Nys30,ChrW01,Woo14}, basic ideas behind deflation
  preconditioners, and the two-level Schur complement preconditioners
  presented in \cite{GriNY14,LiXS16}.
  In~\cref{sec:GEVP}, we illustrate the difficulties in constructing these two-level preconditioners by analysing
  the eigenvalue problems that must be solved. 
  We show that these difficulties are mainly associated with the clustering of eigenvalues near the 
  origin. Motivated by this analysis, in~\cref{sec:preconditioner} we propose reformulating the approximation problem.}
  
The new formulation leads to well-separated eigenvalues that lie away from the origin,
and this allows randomized methods to be used to compute a deflation subspace.
Our approach guarantees a user-defined upper bound on the expected value of the spectral condition number of the preconditioned matrix.
Numerical results for our new  preconditioner and comparisons with other approaches
are given in \cref{sec:numerical_experiments}.
 Concluding remarks are made in \cref{sec:conclusion}.

\cblue{
Our main contributions are:
\begin{itemize}
  \item an analysis of the eigenvalue problems and solvers presented in \cite{GriNY14,LiXS16};
  \item a reformulation of the eigenvalue problem so that it be efficiently solving using randomized methods;
  \item a new two-level preconditioner for symmetric positive definite systems that we refer to as a 
  two-level Nystr\"om--Schur preconditioner;
  \item theoretical bounds on the expected value of the spectral condition number of the preconditioned system.
\end{itemize}
}
%



\medskip
\paragraph{Test environment}
In this study, to demonstrate our theoretical and practical findings, 
we report on numerical experiments using the test matrices  given in~\cref{tab:test_matrices}.
\cblue{This set was chosen to include 2D and 3D problems having a range of densities and with relatively large
spectral condition numbers. In the Appendix, results are given for a much larger set of matrices. 
For each test, the entries of the right-hand side vector $f$ are taken to be random numbers in the interval $[0,1]$.}
All experiments are performed using Matlab 2020b.
\begin{table}
  \label{tab:test_matrices}
  \centering
  \begin{adjustbox}{width=\linewidth}
  \begin{tabular}{lclclclclclclcl}
    \toprule
    Identifier   && $n$              && $nnz(A)$            && $\kappa(A)$   && $n_\Gamma$  && 2D/3D && Application        && Source \\
    \midrule                                                                                                                
    bcsstk38     && \phantom{0}8,032 && \phantom{0,}355,460  && 5.5e+16     && 2,589       && 2D    && Structural problem && SSMC  \\
    ela2d        &&           45,602 && \phantom{0,}543,600  && 1.5e+8      && 4,288       && 2D    && Elasticity problem && FF++  \\
    ela3d        && \phantom{0}9,438 && \phantom{0,}312,372  && 4.5e+5      && 4,658       && 3D    && Elasticity problem && FF++  \\
    msc10848     &&           10,848 &&           1,229,776 && 1.0e+10     && 4,440       && 3D    && Structural problem && SSMC  \\
    nd3k         && \phantom{0}9,000 &&           3,279,690 && 1.6e+7      && 1,785       && 3D    && Not available      && SSMC  \\
    s3rmt3m3     && \phantom{0}5,357 && \phantom{0,}207,123  && 2.4e+10     && 2,058       && 2D    && Structural problem && SSMC  \\
    \bottomrule
  \end{tabular}
  \end{adjustbox}
  \vspace{0.1in}
  \caption{Set of test matrices. $n$ and $nnz(A)$ denote the order of $A$
  and the number of nonzero entries in $A$ disregarding, 
  $\kappa(A)$ is the spectral condition number, $n_\Gamma$ is the order  of the Schur complement \cref{eq:schur_complement}. 
  SSMC refers to SuiteSparse Matrix Collection \cite{DavH11}. FF++ refers to FreeFem++ \cite{Hec12}.}
\end{table}

\medskip
\paragraph{Notation}
Throughout this article, matrices are denoted using uppercase letters; scalars and vectors are lowercase.
The pseudo inverse of a matrix $C$ is denoted by $C^\dagger$ and its transpose is given by $C^\top$.
$\Lambda(M)$ denotes the spectrum of the matrix $M$ and $\kappa(M)$ denotes its condition number.
$\Lambda_k = diag(\lambda_1, \ldots, \lambda_k)$ denotes a $k \times k$ diagonal
matrix with entries on the diagonal equal to $\lambda_1, \ldots, \lambda_k$.
$\widetilde{S}$ (with or without a subscript or superscript) is used as an approximation to a Schur complement matrix.
$\calP$ (with or without a subscript) denotes a (deflation) preconditioner.
$\calM$ (with or without a subscript) denotes a two-level (deflation) preconditioner.
Matrices with an upper symbol such as $\widetilde{Z}$, $\widehat{Z}$, and $\breve{Z}$ 
denote approximations  of the matrix $Z$.
Euler's constant is denoted by $e$.

\section{Background}
\label{sec:background}
\cblue{We start by presenting a brief review of Nystr\"om's  method
for computing a low rank approximation to a matrix
and then recalling key ideas behind two-level preconditioners;  both are required
in later sections.
\subsection{Nystr\"om's method}
\label{sec:Nystrom_background}
Given a matrix $G$, the Nystr\"om approximation of a SPSD matrix $B$ is defined to be
\begin{equation}
  \label{eq:NystromApprox}
  BG (G^\top BG)^\dagger (BG)^\top.
\end{equation}
We observe that there are a large number of variants 
based on different choices of  $G$ (for example, \cite{HalMT11,MarT20,Nak20}).
For $q \ge 0$, the $q$-power iteration Nystr\"om method is obtained by choosing
\begin{equation}
  \label{eq:power}
G = B^q \Omega,
\end{equation}
for a given 
(random) starting matrix $\Omega$. Note that, in practice, for stability it is 
normally necessary to orthonormalize
the columns between applications of $B$.

The variant of Nystr{\"o}m's method we employ is outlined in \cref{alg:nystrom}.
It gives a near-optimal low rank approximation to $B$  and is 
particularly effective when the eigenvalues of $B$ decay rapidly 
after the $k$-th eigenvalue \cite{HalMT11,Nak20}. It requires only one matrix-matrix product
with $B$ (or $q+1$ products if \cref{eq:power} is used).
The rank of the resulting approximation  is $\min(r,k)$, 
where $r$ is the rank of $D_1$, see Step~\ref{line:EVC}.

\begin{algorithm}
  \caption{Nystr{\"o}m's method for computing a low rank approximation to a SPSD matrix.}
  \label{alg:nystrom}
  \begin{algorithmic}[1]
    \Require{A SPSD matrix $B\in\mathbb{R}^{n\times n}$, the required rank $k>0$, 
    an oversampling parameter $p\ge0$ such that $k ,p  \ll n$, and a threshold $\varepsilon$.}
    \vspace{0.05in}
    \Ensure{$\widetilde{B}_k = \widetilde{U}_k \widetilde{\Sigma}_k \widetilde{U}_k^\top \approx B$ where $\widetilde{U}_k$ is orthonormal $\widetilde{\Sigma}_k$ is
    diagonal with non negative entries.}
     \vspace{0.05in}
    \State{Draw a random matrix $G\in \mathbb{R}^{n\times (k+p)}$.}
     \State{Compute  $F = BG$.} 
    \State{Compute the QR factorization $F = QR$.} 
    \State{Set $C = G^\top F$.}
    \State{Compute the EVD $C= V_1 D_1 V_1^\top + V_2 D_2 V_2^\top$,
    where $D_1$ contains all the eigenvalues that are at least $\varepsilon$.}\label{line:EVC}
    \State{Set $T = RV_1 D_1^{-1} (RV_1)^\top$.}
    \State{Compute the EVD $T = W E W^\top$.}\label{line:EVT}
    \vspace{0.1cm}
    \State{Set $\widetilde{U} = QW$, $\widetilde{U}_k = \widetilde{U}(:,1:k)$, $\widetilde{\Sigma} = E(1:k, 1:k)$, 
    and $\widetilde{B}_k = \widetilde{U}_k \widetilde{\Sigma}_k \widetilde{U}_k^\top$.}\label{line:app}
  \end{algorithmic}
\end{algorithm}

Note that, if the eigenvalues are ordered in descending order,
the success of Nystr\"om's method is closely related to the ratio of the $(k+1)$th and the $k$th eigenvalues.
If the ratio is approximately equal to one, $q$ must be large to obtain a good approximation \cite{Sai19}.
}

\subsection{Introduction to two-level preconditioners}
\label{sec:two_level_prec}

Consider the linear system \cref{eq:linear_system}. 
As already noted, deflation techniques are typically used to shift isolated clusters of small eigenvalues 
to obtain a tighter spectrum and a smaller condition number. Such changes
have a positive effect on the convergence of Krylov subspace methods.
Consider the general (left) preconditioned system
\begin{equation}
  \label{eq:precond system}
  \calP A x = \calP b, \qquad \calP \in \R^{n\times n}.
\end{equation}
Given a projection subspace matrix $Z \in \R^{n\times k}$ of full rank and $k \ll n$, 
define the nonsingular matrix $E = Z^\top A Z\in \R^{k\times k}$ and the matrix  $Q = Z E^{-1} Z^\top\in \R^{n\times n}$.
The deflation preconditioner $\calPD\in \R^{n\times n}$ is  defined to be \cite{FraV01}
\begin{equation}
  \label{eq:deflation_prec}
  \calPD = I - AQ.
\end{equation}
It is straightforward to show that
$\calPD$ is a projection matrix and $\calPD A$ has $k$ zero eigenvalues (see \cite{TanNVE09}
for basic properties of $\calPD$).
To solve \cref{eq:linear_system}, we write
\begin{align*}
x = (I - \calPD^\top)x + \calPD^\top x .
\end{align*}
Since $Q$ is symmetric, $\calPD^\top = I - QA$, and so 
\begin{align*}
x =  QAx  + \calPD^\top x = Qb  + \calPD^\top x, 
\end{align*}
and we only need to compute $\calPD^\top x$.
We first find $y$ that satisfies the deflated system
\begin{equation} \label{eq:deflated system}
\calPD Ay =  \calPD b,
\end{equation}
then (due to the identity $A \calPD^\top = \calPD A$) we have that $\calPD^\top y = \calPD^\top x$. We therefore obtain the unique solution
$x = Qb + \calPD^\top y.$
The deflated system \cref{eq:deflated system} is singular and can only be solved using CG if it is consistent \cite{Kaa88}, which is the case here since
the same projection is applied to both sides of a consistent  nonsingular system (\ref{eq:linear_system}).
The deflated system can also be solved using a
preconditioner, giving a two-level preconditioner for the original system.

Tang {\it et al.}~\cite{TanNVE09} illustrate
that rounding errors can result in erratic and slow convergence of CG using $\calPD$.
They thus also consider an adapted deflation preconditioner 
\begin{equation}
  \label{eq:adapted_deflation_prec}
  \calPAD = I - QA + Q,
\end{equation}
that combines $\calPD^\top$ with $Q$.
In exact arithmetic, both  $\calPD$ and $\calPAD$ used with  CG generate the same iterates.
However, numerical experiments \cite{TanNVE09} show that the latter 
is more robust and leads to better numerical behavior of CG\footnote{In \cite{TanNVE09},
$\calPD$ and $\calPAD$ are termed ${\mathcal{P}_{\text{\tiny{DEF1}}}}$ and ${\mathcal{P}_{\text{\tiny{A-DEF2}}}}$, respectively}.

Let $\lambda_n \ge \cdots \ge \lambda_1 > 0$ be the eigenvalues of $A$ with associated normalized eigenvectors $v_n,\ldots, v_1$. For the ideal deflation preconditioner, $\calP_{\footnotesize{\text{ideal}}}$,
the deflation subspace is the invariant subspace spanned by the eigenvectors associated 
with the smallest eigenvalues.
To demonstrate how $\calP_{\footnotesize{\text{ideal}}}$
modifies the spectrum of the deflated matrix, set $Z_k = [v_{1}, \ldots, v_k]$
to be the $n \times k$ matrix whose columns are the eigenvectors corresponding to the smallest eigenvalues.
It follows that $E= Z^\top A Z$ is equal to $\Lambda_k = diag(\lambda_{1}, \ldots, \lambda_k)$ and the preconditioned matrix is
given by
\begin{align*}
  \calP_{\footnotesize{\text{ideal}}}A  &= A - Z_k \Lambda_k Z_k^\top.
\end{align*}
Since $Z_k$ is orthonormal and its columns span an invariant subspace, 
the spectrum of $\calP_{\footnotesize{\text{ideal}}}A$ is $\{\lambda_n, \ldots, \lambda_{k + 1}, 0\}$.
Starting with $x_0$ such that $Z_k^\top r_0 = 0$ ($r_0$ is the initial residual), for $l\ge0$, 
$Z_k^\top (\calP_{\footnotesize{\text{ideal}}} A)^l r_0 = 0$ and $Z_k^\top A^l r_0 = 0$.
Hence the search subspace generated by the preconditioned CG (PCG) method lies in the invariant subspace
spanned by $v_n, \ldots, v_{k+1}$, which is orthogonal to the subspace spanned by the columns of $Z_k$.
Consequently, the effective spectrum of the operator that PCG sees is $\{\lambda_n, \ldots, \lambda_{k+1}\}$ 
and the associated {\em effective spectral condition number} is 
\[ \kappa_{\text{eff}}(\calP_{\footnotesize{\text{ideal}}}A) =\lambda_n/\lambda_{k+1}.\]
Using similar computations, the ideal adapted deflated system is given by:
\begin{equation} \label{eq:ideal adapted}
  \calP_{\footnotesize{\text{A-ideal}}} = A - Z_k \Lambda_k^{-1} Z_k^\top + Z_k Z_k^\top.
\end{equation}
Furthermore, the spectrum of the operator that PCG sees is $\{\lambda_n, \ldots, \lambda_{k+1}, 1, \ldots, 1\}$ 
and the associated effective spectral condition number is 
\[\kappa_{\text{eff}}(\calP_{\footnotesize{\text{A-ideal}}}A) = \max\{1,\lambda_n\}/\min\{1,\lambda_{k+1}\}.\]

In practice, only an approximation of the ideal deflation subspace spanned by the columns of $Z_k$
is available.
Kahl and Rittich~\cite{KahR17} analyze the deflation preconditioner using $\widetilde{Z}_k \approx Z_k$
and present an upper bound on the corresponding effective spectral condition number of the deflated matrix $\kappa\left(\calP A\right)$.
Their bound \cite[Proposition 4.3]{KahR17}, which depends on $\kappa(A)$, 
$\kappa_{\text{eff}}(\calP_{\footnotesize{\text{ideal}}}A)$, 
and the largest principal angle $\theta$ between $\widetilde{Z}_k$ and $Z_k$, is
given by 
\begin{equation}
  \label{eq:Kahl_cond_bound}
  \kappa\left(\calP A\right) \le \left(\sqrt{\kappa(A)} \sin \theta + 
  \sqrt{\kappa_{\text{eff}}(\calP_{\footnotesize{\text{ideal}}}A)}\right)^2,
\end{equation}
where $\sin \theta = \|Z_k Z_k^\top - \widetilde{Z}_k \widetilde{Z}_k^\top\|_2$.

%

\subsection{Schur Complement Preconditioners}
\label{sec:schur_complement_preconditioner}
This section reviews the Schur complement preconditioner with a focus on two-level variants
that were introduced in \cite{GriNY14,LiXS16}.

One-level preconditioners may not provide the required robustness when used with a Krylov 
subspace method because they typically fail to 
capture information about the eigenvectors corresponding to the smallest eigenvalues.
To try and remedy this, \cblue{in their (unpublished) report,} Grigori {\it et al.}~\cite{GriNY14}
and, independently, Li {\it et al.}~\cite{LiXS16} propose  a two-level preconditioner based on 
using a block factorization and approximating the resulting Schur complement. 

Applying graph partitioning techniques (for example, using the  METIS package \cite{KarK97,metis:2020}), 
$A$  can be symmetrically permuted to the $2 \times 2$ doubly bordered block diagonal form
\begin{equation}\label{eq:perm_A}
 P^\top A P = \begin{pmatrix} A_{I} & A_{I\Gamma} \\ A_{\Gamma I} & A_{\Gamma} \end{pmatrix},
\end{equation}
where $A_{I} \in \R^{n_{I} \times n_{I}}$ is a block diagonal matrix,
$A_{\Gamma} \in \R^{n_{\Gamma} \times n_{\Gamma}}$,
$A_{\Gamma I} \in \R^{n_{\Gamma I} \times n_{\Gamma}}$ and $A_{I\Gamma} = A_{\Gamma I}^\top$.
For simplicity of notation, we assume that $A$ is of
the form \cref{eq:perm_A} (and omit the permutation  $P$ from the subsequent discussion).

The block form \cref{eq:perm_A} induces a block LDLT factorization
\begin{equation}\label{eq:A_LDLT}
  A = \begin{pmatrix} I & \\ A_{\Gamma I} A_{I}^{-1} & I \end{pmatrix} \begin{pmatrix} A_{I} & \\ & S_\Gamma \end{pmatrix}
  \begin{pmatrix} I & A_{I}^{-1} A_{I\Gamma}\\ & I \end{pmatrix},
\end{equation}
where 
\begin{equation}\label{eq:schur_complement}
  S_\Gamma = A_{\Gamma} - A_{\Gamma I} A_{I}^{-1} A_{I\Gamma}
\end{equation}
is the Schur complement of $A$ with respect to $A_{\Gamma}$.
Provided the blocks within $A_{I}$ are small, they can be  factorized cheaply  in parallel
using a direct algorithm (see, for example, \cite{Sco01}) and thus we assume that
solving linear systems with  $A_{I}$ is not computationally expensive. 
However, the SPD Schur complement $S_\Gamma$ is typically large and significantly denser than $A_{\Gamma}$
(its size increases with the number of blocks in $A_{I}$)
and, in large-scale practical applications, it may not be possible to explicitly assemble or factorize it.

Preconditioners may be derived by approximating $S_\Gamma^{-1}$. 
An approximate block factorization of $A^{-1}$ is
\begin{align*}
  M^{-1} = \begin{pmatrix} I & -A_I^{-1} A_{I \Gamma I} \\ & I \end{pmatrix} \begin{pmatrix} A_{I}^{-1} & \\ & \widetilde{S}^{-1} \end{pmatrix}
  \begin{pmatrix} I & \\ -A_{\Gamma I}A_{I}^{-1} & I \end{pmatrix},
\end{align*}
where $\widetilde{S}^{-1} \approx S_\Gamma^{-1}$.
If $M^{-1}$ is employed as a preconditioner for $A$ then
the preconditioned system is given by
\begin{equation}\label{eq:M_inv_A}
  M^{-1}A = \begin{pmatrix} I & A_{I}^{-1} A_{I\Gamma} (I-\widetilde{S}^{-1}S_\Gamma) \\ & \widetilde{S}^{-1}S_\Gamma \end{pmatrix},
\end{equation}
with $\Lambda(M^{-1}A) = \{\lambda \in \Lambda(\widetilde{S}^{-1}S_\Gamma)\} \cup \{1\}$.
Thus, to bound the condition number $\kappa(M^{-1}A)$, we need to construct
$\widetilde{S}^{-1}$ so that $\kappa(\widetilde{S}^{-1}S_\Gamma)$ is bounded.
Moreover, \cref{eq:M_inv_A} shows that applying the preconditioner
requires the efficient solution of linear systems with $\widetilde{S}^{-1} S_\Gamma$ and $A_I$, the latter being relatively inexpensive.
We therefore focus on constructing preconditioners $\widetilde S^{-1}$ for linear systems
of the form
\begin{equation}
  \label{eq:Sx=b}
  S_\Gamma w = f.
\end{equation}

Consider the first-level preconditioner obtained by setting
\begin{equation}
 \label{eq:S1}
 \widetilde{S}_1^{-1}:= A_\Gamma^{-1}. 
\end{equation}
Assume for now that we can factorize $A_\Gamma$, although in practice it may
be very large and a recursive construction of the preconditioner  may then be needed (see \cite{XiLS16}).
Let the eigenvalues of the generalized eigenvalue problem
\begin{equation}
  \label{eq:GEVP}
  S_\Gamma z = \lambda \widetilde{S}_1 z
\end{equation}
be  $\lambda_{n_\Gamma} \geq \cdots \geq \lambda_{1} > 0$.
From \cref{eq:schur_complement}, $\lambda_{n_\Gamma} \leq 1$ and so
$$\kappa(\widetilde{S}_1^{-1}S_\Gamma) = \frac{\lambda_{n_\Gamma}}{\lambda_{1}} \le \frac{1}{\lambda_{1}}.$$
As this is unbounded as $\lambda_{1}$ approaches zero, we seek
to add a low rank term to ``correct'' the approximation  and shift the smallest $k$ eigenvalues 
of $\widetilde{S}_1^{-1}S_\Gamma$.
Let $\Lambda_k = diag\{\lambda_{1}, \ldots ,\lambda_{k}\}$ and let
$Z_k\in \R^{n_\Gamma \times k}$ be the matrix whose columns are the corresponding eigenvectors. 
Without loss of generality, we assume $Z_k$ is $A_\Gamma$-orthonormal. 
Let the Cholesky factorization of $A_\Gamma$ be 
\begin{equation}\label{eq:cholesky}
A_\Gamma =R_\Gamma^\top R_\Gamma
\end{equation}
and  define
\begin{equation}
  \label{eq:S2}
  \widetilde{S}_2^{-1} := A_\Gamma^{-1} + Z_k (\Lambda_k^{-1} - I)  Z_k^\top.
\end{equation}
$\widetilde{S}_2^{-1}$ is an additive combination of the first-level preconditioner 
$\widetilde{S}_1^{-1}$ 
and an adapted deflation preconditioner associated with the subspace spanned by the columns of $U_k = R_\Gamma Z_k$, 
which is an invariant subspace of \cblue{$R_\Gamma^{-1} S_\Gamma R_\Gamma^{-\top}$.}
Substituting $U_k$ into \cref{eq:S2} and using \cref{eq:cholesky},  
\begin{equation}
  \label{eq:tildeS_def}
  \widetilde{S}_2^{-1} = R_\Gamma^{-1} (I + U_k (\Lambda_k^{-1} - I)  U_k^\top) R_\Gamma^{-\top}.
\end{equation}
Setting $Q = U_k \Lambda_k^{-1} U_k^\top$ in \cref{eq:adapted_deflation_prec} gives
\cblue{
\begin{align*}
  \calPAD &= R_\Gamma \widetilde{S}_2^{-1} R_\Gamma^\top.
\end{align*}
  }
Now $\widetilde{S}_2^{-1} S_\Gamma= R_\Gamma^{-1}\calPAD R_\Gamma^{-\top} S_\Gamma $ 
and $\calPAD R_\Gamma^{-\top} S_\Gamma R_\Gamma^{-1}$ are spectrally equivalent and
$\Lambda(\widetilde{S}_2^{-1}S_\Gamma)= \{ \lambda_{n_\Gamma}, \lambda_{n_\Gamma - 1}, ..., \lambda_{k + 1\}} \cup \{1\}$.
It follows that 
$$\kappa(\widetilde{S}_2^{-1}S_\Gamma) = \frac{\lambda_{n_\Gamma}}{\lambda_{k + 1}} \le \frac{1}{\lambda_{k + 1}}.$$


Grigori {\it et al.}~\cite{GriNY14} note that \cref{eq:GEVP} is equivalent to the generalized eigenvalue problem
\begin{equation}
  \label{eq:GEVP_2}
 (A_{\Gamma} - S_{\Gamma})z =  A_{\Gamma I} A_I^{-1} A_{I\Gamma} z = \sigma A_\Gamma z, \qquad \sigma = 1-\lambda.
\end{equation}
Setting $u = R_\Gamma z$ and defining
\begin{equation} \label{eq:H}
H = R_\Gamma^{-\top} A_{\Gamma I} A_I^{-1} A_{I\Gamma} R_\Gamma^{-1}, 
\end{equation}
\cref{eq:GEVP_2} becomes
\begin{equation}
  \label{eq:GEVP_3}
  H u = \sigma u.
\end{equation}
Thus, the smallest eigenvalues $\lambda$ of \cref{eq:GEVP} are transformed to the largest eigenvalues $\sigma$
of problems  \cref{eq:GEVP_2,eq:GEVP_3}. 
Grigori {\it et al.} employ a randomized algorithm to compute  a low rank eigenvalue decomposition (EVD) of $H$
that approximates its largest eigenvalues and vectors, which
are multiplied by $R_\Gamma^{-1}$ 
to obtain approximate eigenvectors of $A_\Gamma^{-1} S_\Gamma$.

In~\cite{LiXS16}, Li {\it et al.} write the inverse of the Schur complement $S_\Gamma$ as:
\begin{align}
  \label{eq:S-1_2}
  \begin{split}
    S_\Gamma^{-1} &= \left( A_\Gamma - A_{\Gamma I} A_I^{-1} A_{I \Gamma}\right)^{-1} \\
                  &= \left( R_\Gamma^\top R_\Gamma - A_{\Gamma I} A_I^{-1} A_{I \Gamma}\right)^{-1} \\
                  &= R_\Gamma^{-1} \left(I - H\right)^{-1} R_\Gamma^{-\top},
  \end{split}
\end{align}
where the symmetric positive semidefinite (SPSD) matrix $H$ is given by \cref{eq:H}.
Since $I-H = R_\Gamma^{-\top} S_\Gamma R_\Gamma^{-1}$ is SPD, 
the eigenvalues $\sigma_1 \ge \ldots \ge \sigma_{n_\Gamma}$ of $H$ belong to $[0, 1]$.
Let the EVD of $H$ be 
$$H = U \Sigma U^\top,$$ 
where $U$ is orthonormal and 
$\Sigma = diag\{\sigma_1, \ldots, \sigma_{n_\Gamma}\}$. It follows that 
\begin{align}
  \label{eq:S_1_3}
  \begin{split}
    S_\Gamma^{-1} &= R_\Gamma^{-1} \left(I - U\Sigma U^\top\right)^{-1} R_\Gamma^{-\top} \\
                  &= R_\Gamma^{-1} U \left( I - \Sigma \right)^{-1} U^\top R_\Gamma^{-\top} \\
                  &= R_\Gamma^{-1} \left(I + U \left( \left( I - \Sigma \right)^{-1} - I\right) U^\top \right) R_\Gamma^{-\top} \\
                  &= A_\Gamma^{-1} + R_\Gamma^{-1} U \left( \left( I - \Sigma \right)^{-1} - I\right) U^\top R_\Gamma^{-\top}.
  \end{split}
\end{align}
If $H$ has an approximate EVD of the form 
\begin{align*}
 H \approx U\widetilde \Sigma U^\top, \qquad \widetilde\Sigma = diag\{\widetilde{\sigma}_1, \ldots, \widetilde{\sigma}_{n_\Gamma}\},
\end{align*}
then an approximation of $S_\Gamma^{-1}$ is
\begin{align}
  \label{eq:Stilde-1}
  \begin{split}
    \widetilde{S}^{-1}= A_\Gamma^{-1} + R_\Gamma^{-1} U \left( \left( I - \widetilde{\Sigma} \right)^{-1} - I\right) U^\top R_\Gamma^{-\top}.
  \end{split}
\end{align}
The simplest selection of $\widetilde \Sigma$ is the one that ensures the $k$ largest eigenvalues of $(I-\widetilde \Sigma)^{-1}$
match the largest eigenvalues of $(I-\Sigma)^{-1}$.
 Li {\it et al.} set $\widetilde{\Sigma} = diag(\sigma_1, \ldots, \sigma_k, \theta, \ldots, \theta)$, 
 where $\theta \in [0, 1]$. 
 The resulting preconditioner can be written  as
\begin{equation}
  \label{eq:precLiXS16}
  \widetilde{S}_{\theta}^{-1} = \frac{1}{1-\theta} A_\Gamma^{-1} + Z_k \left( \left(I - \Sigma_k \right)^{-1} - \frac{1}{1-\theta} I\right) Z_k^\top,
\end{equation}
where 
\cblue{
  $\Sigma_k = diag(\sigma_1, \ldots, \sigma_k)$ and
}
the columns of $Z_k = R_\Gamma^{-1}U_k$  are the eigenvectors  
 corresponding to the $k$ largest eigenvalues of $H$. In \cite{LiXS16}, it is shown that  
$\kappa(\widetilde{S}_{\theta}^{-1}S) = (1-\sigma_{n_{\Gamma}})/(1-\theta)$, 
which takes its minimum value for $\theta = \sigma_{k+1}$.

\cblue{
In the next section, we analyse the eigenvalue problems that need 
to be solved to construct the preconditioners \cref{eq:S2,eq:precLiXS16}.
In particular, we show that the approaches presented in \cite{GriNY14,LiXS16} 
for tackling these problems are inefficient because of the eigenvalue distribution.
}
\section{Analysis of $Hu = \sigma u$}
\label{sec:GEVP}
\subsection{Use of the Lanczos method}
\label{sec:nonefficientLanczos}
Consider the  eigenproblem:
\begin{align*}
  \begin{split}
    & \text{Given } \varepsilon>0, \text{find all the eigenpairs } (\lambda, z) \in \R \times \R^{n_\Gamma} \text{ such that}\\
    & S_\Gamma z = \lambda A_\Gamma z, \qquad \lambda < \varepsilon.
  \end{split}
\end{align*}
This can be rewritten as:
\begin{align}
  \label{eq2:GEVP}
  \begin{split}
    & \text{Given } \varepsilon>0, \text{find all the eigenpairs } (\lambda, z) \in \R \times \R^{n_\Gamma} \text{ such that}\\
    & (I - H) u = \lambda u, \qquad z = R_\Gamma^{-1} u, \qquad \lambda < \varepsilon,
  \end{split}
\end{align}
where $R_\Gamma$ and $H$ are given by \cref{eq:cholesky,eq:H}.
Consider also the eigenproblem:
\begin{align}
  \label{eq2:GEVP_3}
  \begin{split}
    & \text{Given } \varepsilon>0, \text{find all the eigenpairs } (\sigma, u) \in \R \times \R^{n_\Gamma} \text{ such that}\\
    & H u = \sigma u, \qquad \sigma > 1 - \varepsilon.
  \end{split}
\end{align}
As already observed, each eigenpair $(\lambda, z)$ of \cref{eq2:GEVP} corresponds 
to the eigenpair  $(1-\lambda, R_\Gamma z)$ of \cref{eq2:GEVP_3}.
Consider using the Lanczos method to solve these eigenproblems.
The Krylov subspace at iteration $j$ generated for \cref{eq2:GEVP} is 
$$K_j((I - H), v_1) = \text{span}(v_1, (I-H) v_1, \ldots, (I-H)^{j-1}v_1),$$
while the subspace generated for \cref{eq2:GEVP_3} is
$$K_j(H, v_1) = \text{span}(v_1, H v_1, \ldots, H^{j-1}v_1).$$
It is clear that, provided the same starting vector $v_1$ is used, $K_j((I - H), v_1)$ and $K_j(H, v_1)$ are identical. 
Suppose that $[\mathcal{V}_{j}, v_{j+1}]$ is the output of the Lanczos basis of the 
Krylov subspace, then the subspace relations that hold at iteration $j$ are
$$(I-H) \mathcal{V}_j = \mathcal{V}_j T_j + v_{j+1} h_j^\top,$$
$$H\mathcal{V}_j = \mathcal{V}_j (I - T_j) - v_{j+1} h_j^\top,$$
where $T_j \in \R^{j \times j}$ is a symmetric tridiagonal matrix and $h_j \in \R^{j}$.
The eigenpair $(\lambda, z)$ (respectively, $(\sigma, u)$) corresponding to the smallest (respectively, largest) 
eigenvalue in \cref{eq2:GEVP} (respectively, \cref{eq2:GEVP_3}) is  
approximated by the eigenpair $(\widetilde{\lambda},R_\Gamma^{-1}\mathcal{V}_j \widetilde{u})$ 
(respectively, $(\widetilde{\sigma},\mathcal{V}_j \widetilde{u})$) corresponding to 
the smallest (respectively, largest) eigenvalue of $T_j$ (respectively, $I - T_j$).
To overcome memory constraints,  the Lanczos procedure is typically restarted after a 
chosen number of iterations, at each restart discarding 
the non convergent part of the Krylov subspace \cite{Ste02}.
Hence, starting with the same $v_1$ and performing the 
same number of iterations per cycle, in exact arithmetic
the accuracy obtained when solving \cref{eq2:GEVP,eq2:GEVP_3} is identical.

\cblue{
  Having shown that the convergence of Lanczos' method for solving~\cref{eq2:GEVP,eq2:GEVP_3} is
  the same, we focus on \cref{eq2:GEVP_3}.
}
\begin{figure}[htbp]
\centering
  \label{fig:clustered_eigenvalues}
    \includegraphics[height=5cm]{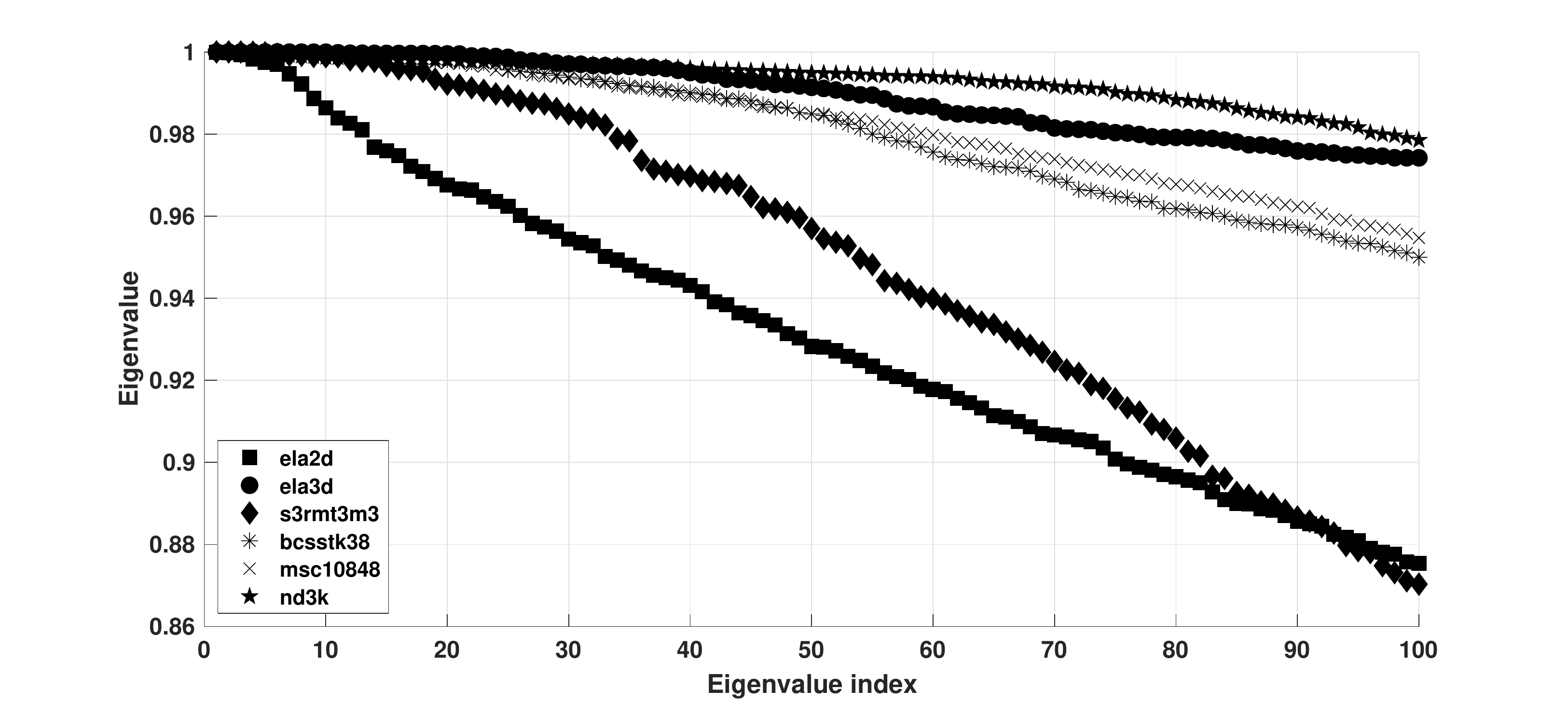}
  \caption{Largest 100 eigenvalues of $H = R_\Gamma^{-\top}A_{\Gamma I} A_{I}^{-1} A_{I\Gamma} R_\Gamma^{-1}$ 
  associated with our test matrices computed to an accuracy of $10^{-8}$ using the Krylov-Schur
  method \cite{Ste02}.}
\end{figure}
In~\cref{fig:clustered_eigenvalues},  for each of our test matrices in \cref{tab:test_matrices} we plot the 100 largest 
eigenvalues of the matrix  $H$ given by \cref{eq:H}.
We see that the largest eigenvalues (which are the ones
that we require) are clustered near one \cblue{and they do not decay rapidly}. As there are a significant 
number of eigenvalues in the cluster, computing \cblue{the largest $k$ (for $k = O(10)$)}  and
the corresponding eigenvectors with sufficient accuracy using the Lanczos method is challenging.
\cblue{Similar distributions were observed for the larger test set
that we report on in the Appendix, particularly for problems for which
the one-level preconditioner $\widetilde{S}_1$ was found to perform poorly, which is generally the case when 
$\kappa(A)$ is large.}
\Cref{tab:clustered_eigenvalues} reports the Lanczos iteration counts ($it_{\text{Lan}}$) for computing 
the $k=20$ and $40$ largest eigenpairs (that is, the number of linear systems that are solved
in the Lanczos method). 
In addition, we present the PCG iteration count ($it_{\text{PCG}}$) for  solving the linear system \cref{eq:Sx=b}
using the first-level preconditioner $\widetilde{S}_{1} = A_{\Gamma}^{-1}$ and  the two-level
preconditioner $\widetilde{S}_2$ given by \cref{eq:S2}. 
\begin{table}[htbp]
  \label{tab:clustered_eigenvalues}
  \centering
  {
\footnotesize
  \begin{tabular}{lrcrrrcrrr}
    \toprule
                 &                     &&  \multicolumn{7}{c}{$\widetilde{S}_{2}$}      \\
                 \cmidrule(lr){4-10}
                 & $\widetilde{S}_{1}$ &&  \multicolumn{3}{c}{$k = 20$}                       && \multicolumn{3}{c}{$k = 40$}                    \\
    \cmidrule(lr){4-7}\cmidrule(lr){8-10}
    Identifier      & $it_{\text{PCG}}$   && $it_{\text{Lan}}$ & $it_{\text{PCG}}$ & total       && $it_{\text{Lan}}$ & $it_{\text{PCG}}$ & total     \\
                 \midrule

    bcsstk38     & 584                 && 797               & 122               & 919         && 730               & 67                & 797       \\
    el2d         & 914                 && 1210              & 231               & 1441        && 982               & 120               & 1102      \\
    el3d         & 174                 && 311               & 37                & 348         && 389               & 27                & 416       \\
    msc10848     & 612                 && 813               & 116               & 929         && 760               & 63                & 823       \\
    nd3k         & 603                 && 1796              & 143               & 1939        && 1349              & 105               & 1454       \\
    s3rmt3m3     & 441                 && 529               & 70                & 599         && 480               & 37                & 517       \\
    \bottomrule
  \end{tabular}
  }
    \vspace{0.1in}
  \caption{The Lanczos iteration count ($it_{\text{Lan}}$) and the iteration count 
  for PCG ($it_{\text{PCG}}$). The convergence tolerance for the Lanczos method
  and PCG is $10^{-6}$. The size of the Krylov subspace per cycle is $2k$.}
\end{table}
We see that, in terms of the total iteration count, the first-level preconditioner is the 
more efficient option.
It is of interest to consider whether relaxing the convergence tolerance $\varepsilon_{\text{Lan}}$ in the Lanczos method can reduce
the total iteration count for $\widetilde{S}_2$.
\Cref{tab:clustered_eigenvalues_1problem} illustrates the effect of varying $\varepsilon_{\text{Lan}}$ for problem el3d (results for the 
other test problems are consistent).
Although $it_{\text{Lan}}$ decreases as $\varepsilon_{\text{Lan}}$ increases,  $it_{\text{PCG}}$
increases and the total 
count still exceeds the 175 PCG iterations required by the first-level preconditioner $\widetilde{S}_1$.
\begin{table}[htbp]
  \label{tab:clustered_eigenvalues_1problem}
  \centering
    {
\footnotesize
  \begin{tabular}{lrrrcrrr}
    \toprule
                               &  \multicolumn{3}{c}{$k = 20$}                  && \multicolumn{3}{c}{$k = 40$} \\
    \cmidrule(rl){2-4}\cmidrule(lr){6-8}
    
    $\varepsilon_{\text{Lan}}$ & $it_{\text{Lan}}$ & $it_{\text{PCG}}$ & total  && $it_{\text{Lan}}$ & $it_{\text{PCG}}$ & total    \\
      \midrule

    $0.1$                      & 50                & 131               & 181    && 80                & 101               & 181      \\
    $0.08$                     & 50                & 131               & 181    && 100               & 85                & 185       \\
    $0.06$                     & 60                & 121               & 181    && 100               & 85                & 185       \\
    $0.04$                     & 82                & 100               & 182    && 120               & 71                & 191       \\
    $0.02$                     & 127               & 64                & 201    && 207               & 37                & 244       \\
    $0.01$                     & 169               & 41                & 210    && 259               & 32                & 291       \\
    $0.005$                    & 213               & 38                & 251    && 316               & 29                & 345       \\
    $0.001$                    & 247               & 37                & 284    && 372               & 28                & 400       \\
        \bottomrule
  \end{tabular}
  }
  \vspace{0.1in}
  \caption{Problem el3d and two-level preconditioner $\widetilde{S}_2$: sensitivity of the number of 
  the Lanczos iteration count ($it_{\text{Lan}}$) and the iteration count for PCG ($it_{\text{PCG}}$) 
  to the convergence tolerance $\varepsilon_{\text{Lan}}$. 
  The PCG convergence tolerance is $10^{-6}$. The size of the Krylov subspace per cycle is $2k$.}
\end{table}

As already observed, in \cite{XiLS16} a recursive (multilevel) scheme is proposed to help
mitigate the computational costs of building and applying the preconditioner.
Nevertheless, the Lanczos method is still used, albeit with reduced costs for applying the operator matrices.

\subsection{Use of Nystr{\"o}m's method}
\label{sec:randomized_methods}
As suggested  in \cite{GriNY14}, an alternative approach to approximating the dominant subspace of $H$ is to 
use a randomized method, specifically a randomized eigenvalue decomposition.
Because $H$ is SPSD, Nystr\"om's method can be use.
Results are presented in \Cref{tab:nystromH} 
for problem el3d (results for our other test examples are consistent with these).
Here $p$ is the oversampling parameter
\cblue{and $q$ is the power iteration parameter}.
These show that, as with the Lanczos method,  Nystr{\"o}m's method 
struggles to approximate the dominant eigenpairs of $H$.
Using $k=20$ (respectively, 40) exact eigenpairs, PCG using $\widetilde{S}_2$ 
requires 37 (respectively, 28) iterations. To obtain the same
iteration counts using vectors computed using  Nystr\"om's method requires 
the oversampling parameter to be greater than $2000$, which is clearly prohibitive.
\cblue{Using the power iteration improves the quality of the approximate subspace. However, the large value of $q$ needed
to decrease the PCG iteration count means a large number of linear systems must be solved with $A_\Gamma$,
in addition to  the work involved
in the orthogonalization that is needed between the power iterations to maintain stability.
Indeed, it is sufficient to look at \Cref{fig:clustered_eigenvalues} to 
predict this behaviour for any randomized method applied to $H$.
The lack of success of existing strategies motivates us, in the next section,  to
reformulate the eigenvalue problem to one with a spectrum 
that is easy to approximate.}

\begin{table}[htbp]
  \label{tab:nystromH}
  \centering
  {
  \footnotesize
  \begin{tabular}{rcrcr}
    \toprule
    $p$                       &&  $k = 20$             &&     $k = 40$ \\
      \midrule
    $100$                     && 171                   && 169               \\
    $200$                     && 170                   && 165                \\
    $400$                     && 165                   && 161                \\
    $800$                     && 155                   && 146                \\
    $1600$                    && 125                   && 111                \\
    $3200$                    && 55                    && 45                 \\
        \bottomrule
  \end{tabular} \quad \quad 
  \cblue{
  \begin{tabular}{rcrcr}
    \toprule
    $q$                       &&  $k = 20$             &&     $k = 40$   \\
      \midrule
    $0$                       && 172                   &&  171           \\
    $20$                      && 121                   &&  87            \\
    $40$                      && 86                    &&  48            \\
    $60$                      && 68                    &&  34            \\
    $80$                      && 55                    &&  30            \\
    $100$                     && 46                    &&  29            \\
        \bottomrule
  \end{tabular}
  }
  }
  \vspace{0.1in}
  \caption{PCG iteration counts for problem el3d using the two-level preconditioner $\widetilde{S}_2$ constructed 
  using a rank $k$ approximation of $H = R_\Gamma^{-\top}A_{\Gamma I} A_{I}^{-1} A_{I\Gamma} R_\Gamma^{-1}$. 
  The PCG convergence tolerance is $10^{-6}$.
  \cblue{
  Nystr\"om's method applied to $H$ with  the oversampling parameter $p \ge 100$ and the power iteration parameter $q = 0$ (left) and
  with $p = 0$ and  $q \ge 0$ (right).}
  }
\end{table}

\section{Nystr\"om--Schur two-level preconditioner}
\label{sec:preconditioner}

In this section, we propose reformulating the eigenvalue problem to obtain a new one 
such that the desired eigenvectors correspond to the largest eigenvalues and
\cblue{these eigenvalues} are well separated from the remaining eigenvalues: this is what is needed for
randomized methods to be successful.

\subsection{Two-level preconditioner for $S_\Gamma$}
\label{sec:two_level for S}

Applying the Sherman Morrison Woodbury identity \cite[2.1.3]{GolV96},
the inverse of the Schur complement $S_\Gamma$ \cref{eq:schur_complement} can be written as:
\begin{align}
  \label{eq:S-1}
  \begin{split}
    S_\Gamma^{-1} &= A_\Gamma^{-1} + A_\Gamma^{-1} A_{\Gamma I} (A_I - A_{I\Gamma} A_\Gamma^{-1} A_{\Gamma I})^{-1} A_{I \Gamma} A_\Gamma^{-1}\\
                  &= A_\Gamma^{-1} + A_\Gamma^{-1} A_{\Gamma I} S_I^{-1} A_{I \Gamma} A_\Gamma^{-1},
  \end{split}
\end{align}
where 
\begin{equation}
  \label{eq:SI}
  S_I = A_I - A_{I\Gamma} A_\Gamma^{-1} A_{\Gamma I}
\end{equation}
is the Schur complement of $A$ with respect to $A_I$.
Using the Cholesky factorization \cref{eq:cholesky}, we have
\begin{equation}
  \label{eq:Rt_S-1_R}
  R_\Gamma S_\Gamma^{-1} R_\Gamma^{\top} = I + R_\Gamma^{-\top} A_{\Gamma I} S_I^{-1} A_{I \Gamma} R_\Gamma^{-1}.
\end{equation}
\cblue{Note that if $(\lambda, u)$ is an eigenpair of $R_\Gamma^{-\top} S_\Gamma R_\Gamma^{-1}$, 
  then $(\frac{1}{\lambda} - 1, u)$ is  an eigenpair of $R_\Gamma^{-\top} A_{\Gamma I} S_I^{-1} A_{I \Gamma} R_\Gamma^{-1}$.
  Therefore, the cluster of eigenvalues of $R_\Gamma^{-\top} S_\Gamma R_\Gamma^{-1}$ 
  near the origin (which correspond to the cluster of eigenvalues of $H$ near 1) 
  correspond to very large and highly separated eigenvalues of 
  $R_\Gamma^{-\top} A_{\Gamma I} S_I^{-1} A_{I \Gamma} R_\Gamma^{-1}$. 
  Hence, using randomized methods to approximate the dominant subspace of 
  $R_\Gamma^{-\top} A_{\Gamma I} S_I^{-1} A_{I \Gamma} R_\Gamma^{-1}$ 
  can be an efficient way of computing a deflation subspace for $R_\Gamma^{-\top} S_\Gamma R_\Gamma^{-1}$.
}
Now assume that we have a low rank approximation
\begin{equation}
  \label{eq:lr_approximation_S-1_A-I}
  R_\Gamma^{-\top} A_{\Gamma I} S_I^{-1} A_{I \Gamma} R_\Gamma^{-1} \approx \breve{U}_k \breve{\Sigma}_k \breve{U}_k^\top ,
\end{equation}
where $\breve{U}_k \in \R^{n_\Gamma \times k}$ is  orthonormal  and $\breve{\Sigma}_k \in \R^{k \times k}$ is diagonal.
Combining \cref{eq:Rt_S-1_R} and \cref{eq:lr_approximation_S-1_A-I}, we can
define a preconditioner for $R_\Gamma^{-\top} S_\Gamma R_\Gamma^{-1}$ to be
\begin{equation}
  \label{eq:proposed_deflation_prec}
  \calP_1 = I + \breve{U}_k \breve{\Sigma}_k \breve{U}_k^\top.
\end{equation}
The preconditioned matrix $\calP_1 R_\Gamma^{-\top} S_\Gamma R_\Gamma^{-1}$ is 
spectrally equivalent to $R_\Gamma^{-1}\calP_1 R_\Gamma^{-\top} S_\Gamma$. Therefore,
the preconditioned system can be written as 
\begin{equation}
  \label{eq:prec1S}
  \calM_1 S_\Gamma = R_\Gamma^{-1}\calP_1R_\Gamma^{-\top} S_\Gamma = \left(A_\Gamma^{-1} + \breve{Z}_k \breve{\Sigma}_k \breve{Z}_k^\top\right)S_\Gamma,
\end{equation}
where $\breve{Z}_k = R_\Gamma^{-1} \breve{U}_k$.
If  \cref{eq:lr_approximation_S-1_A-I} is 
obtained using a  truncated EVD denoted by $U_k \Sigma_k U_k^\top$, then $\breve{U}_k = U_k$ and the subspace spanned by the 
columns of $U_k$ is an invariant subspace of $R_\Gamma S_\Gamma^{-1} R^\top_\Gamma$ 
and of its inverse $R_\Gamma^{-1} S_\Gamma R^{-\top}_\Gamma$.
Furthermore, using the truncated EVD, \cref{eq:proposed_deflation_prec} is an 
adapted deflation preconditioner for $R_\Gamma^{-\top} S_\Gamma R_\Gamma^{-1}$.
Indeed, as the columns of $U_k$ are orthonormal eigenvectors, we have from \cref{eq:Rt_S-1_R} that
$R_\Gamma S_\Gamma^{-1} R_\Gamma^{\top} U_k = U_k (I + \Sigma_k)$. Hence
$R_\Gamma^{-\top} S_\Gamma R_\Gamma^{-1} U_k = U_k (I + \Sigma_k)^{-1}$  and the preconditioned matrix is
\begin{align*}
  \calPAD R_\Gamma^{-\top} S_\Gamma R_\Gamma^{-1} &= R_\Gamma^{-\top} S_\Gamma R_\Gamma^{-1} + U_k \Sigma_k (I + \Sigma_k)^{-1}U_k^\top\\
                                                  &= R_\Gamma^{-\top} S_\Gamma R_\Gamma^{-1} + U_k \left(\left(I+ \Sigma_k\right)-I\right) (I + \Sigma_k)^{-1} U_k^\top \\
                                                &= R_\Gamma^{-\top} S_\Gamma R_\Gamma^{-1} - U_k (I + \Sigma_k)^{-1} U_k^\top + U_k U_k^\top,
\end{align*}
which has the same form as the ideal adapted preconditioned matrix \cref{eq:ideal adapted}. 

Note that given the matrix $\breve{U}_k$ in the approximation \cref{eq:lr_approximation_S-1_A-I}, 
then following \cref{sec:two_level_prec}, we can define a deflation preconditioner for $R_\Gamma^{-\top} S_\Gamma R_\Gamma^{-1}$.
Setting $E_k = \breve{U}_k^\top R_\Gamma^{-\top} S_\Gamma R_\Gamma^{-1} \breve{U}_k$ and 
$Q= \breve{U}_k E^{-1} \breve{U}_k^\top$,  the deflation preconditioner is
\begin{equation}
  \label{eq:deflation_comp_from_low_rank}
  \calPADi{1-} = I - QR_\Gamma^{-\top} S_\Gamma R_\Gamma^{-1} + Q.
\end{equation}
The preconditioned Schur complement  $\calPADi{1-} R_\Gamma^{-\top} S_\Gamma R_\Gamma^{-1}$ 
is spectrally similar to $R_\Gamma^{-1} \calPADi{1-} R_\Gamma^{-\top} S_\Gamma$
and thus
\begin{equation}
 \label{eq:our_two_level_preconditioner}
  \calMADi{1-} = R_\Gamma^{-1} \calPADi{1-} R_\Gamma^{-\top}
\end{equation}
is a two-level preconditioner for $S_\Gamma$. 

\subsection{Lanczos versus Nystr\"om}
The two-level preconditioner \cref{eq:our_two_level_preconditioner} relies on 
computing a low-rank approximation \cref{eq:lr_approximation_S-1_A-I}.
We now consider the difference between using the Lanczos and Nystr\"om methods for this.

Both methods require the application of $R_\Gamma^{-\top} A_{\Gamma I} S_I^{-1} A_{I \Gamma} R_\Gamma^{-1}$ 
to a set of $k+p$ vectors, where $k > 0$ is the required rank and $p \ge0$.
Because explicitly computing the SPD matrix $S_I = A_I - A_{I\Gamma}A_\Gamma^{-1} A_{\Gamma I}$
and factorizing it is prohibitively expensive, applying $S_I^{-1}$ 
must be done using an iterative solver.

The Lanczos method builds a Krylov subspace of dimension $k+p$ in order to compute a low-rank approximation.
Therefore,  $k+p$ linear systems must be solved, each with one right-hand side, first for $R_\Gamma$, 
then for $S_I$, and then for $R_\Gamma^\top$.
However, the Nystr\"om method requires the solution of only one linear system with 
$k+p$ right-hand sides for $R_\Gamma$, then for $S_I$, and then for $R_\Gamma^\top$.
This allows the use of matrix-matrix operations rather than less efficient matrix-vector operations.
Moreover, as we will illustrate in~\cref{sec:numerical_experiments},
block Krylov subspace methods, such as block CG \cite{Ole80},  for solving the system
with $S_I$ yield faster convergence than their classical counterparts. 
When the Nystr\"om method is used, 
we call the resulting preconditioner \cref{{eq:our_two_level_preconditioner}} the {\it Nystr\"om--Schur preconditioner}.

\subsection{Avoiding computations with $R_\Gamma$}
\label{sec:further_approximation}
For  large scale problems, computing the Cholesky factorization $A_\Gamma = R_\Gamma^\top R_\Gamma$ is prohibitive
and so we would like to avoid computations with $R_\Gamma$. 
\cblue{We can achieve this by using an iterative solver to solve linear systems with $A_\Gamma$. Note that this is possible
when solving the generalized eigenvalue problem \cref{eq:GEVP}.
}
Because $A_\Gamma$ is typically well conditioned, so too is $R_\Gamma$. 
Thus, we can reduce the cost of computing the Nystr\"om--Schur preconditioner by approximating the SPSD matrix 
$A_{\Gamma I} S_I^{-1} A_{I \Gamma}$ 
(or even by approximating $S_I^{-1}$). Of course, this needs to be done 
without seriously adversely affecting the preconditioner quality.
Using an approximate factorization
\begin{equation}
  \label{eq:internal_term_lr_approximation_1}
  A_{\Gamma I} S_I^{-1} A_{I \Gamma} \approx \widetilde{W}_k \widetilde{\Sigma}_k \widetilde{W}_k^\top,
\end{equation}
an alternative deflation preconditioner is 
\begin{align*}
  \calP_2 &= I + R_\Gamma^{-\top} \widetilde{W}_k \widetilde{\Sigma}_k \widetilde{W}_k^\top R_\Gamma^{-1},\\
          &= R_\Gamma^{-\top} \left(A_\Gamma + \widetilde{W}_k \widetilde{\Sigma}_k \widetilde{W}_k^\top \right) R_\Gamma^{-1}.
\end{align*}
The preconditioned Schur complement  $\calP_2 R_\Gamma^{-\top} S_\Gamma R_\Gamma^{-1}$ 
is spectrally similar to $R_\Gamma^{-1} \calP_2 R_\Gamma^{-\top} S_\Gamma$ and, setting $\widetilde{Z}_k = A_\Gamma^{-1} \widetilde{W}_k$,
we have
\begin{equation}
  \label{eq:left_prec_mat}
  \calM_2 S_\Gamma = R_\Gamma^{-1} \calP_2 R_\Gamma^{-\top} S_\Gamma = (A_\Gamma^{-1} + \widetilde{Z}_k \widetilde{\Sigma}_k \widetilde{Z}_k^\top) S_\Gamma.
\end{equation}
Thus $\calM_2 = A_\Gamma^{-1} + \widetilde{Z}_k \widetilde{\Sigma}_k \widetilde{Z}_k^\top$ is a 
variant of the Nystr\"om--Schur preconditioner for $S_\Gamma$
that avoids computing $R_\Gamma$.

Alternatively, assuming we have an approximate factorization
\begin{equation}
  \label{eq:internal_term_lr_approximation_2}
  S_I^{-1} \approx \widehat{V}_k \widehat{\Sigma}_k \widehat{V}_k^\top,
\end{equation}
yields
$$ \calP_3 = I + R_\Gamma^{-\top} A_{\Gamma I}  \widehat{V}_k \widehat{\Sigma}_k \widehat{V}_k^\top A_{I \Gamma} R_\Gamma^{-1}.$$
Again,  $\calP_3 R_\Gamma^{-\top} S_\Gamma R_\Gamma^{-1}$ 
is spectrally similar to $R_\Gamma^{-1} \calP_3 R_\Gamma^{-\top} S_\Gamma$ and, setting
$\widehat{Z}_k = A_\Gamma^{-1} A_{\Gamma I} \widehat{V}_k$, we have
\begin{equation}
  \label{eq:left_prec_mat2}
  \calM_3 S_\Gamma = R_\Gamma^{-1} \calP_3 R_\Gamma^{-\top} S_\Gamma= (A_\Gamma^{-1} + \widehat{Z}_k \widehat{\Sigma}_k \widehat{Z}_k^\top) S_\Gamma,
\end{equation}
which gives another variant of the Nystr\"om--Schur preconditioner.
In a similar way to defining $\calMADi{1-}$ \cref{eq:deflation_comp_from_low_rank}, 
we can define $\calMADi{2-}$ and $\calMADi{3-}$. Note that $\calMADi{2-}$ and $\calMADi{3-}$ also avoid computations with $R_\Gamma$.

\subsection{Nystr\"om--Schur preconditioner}
\Cref{alg:prec1} presents the construction of the Nystr\"om--Schur preconditioner $\calM_2$;
an analogous derivation yields the variant $\calM_3$.
\begin{algorithm}
  \caption{Construction of the Nystr\"om--Schur preconditioner \cref{eq:left_prec_mat}}
  \label{alg:prec1}
  \begin{algorithmic}[1]
    \Require{$A$ in block form \cref{eq:perm_A}, $k>0$ and $p\ge0$ ($k ,p  \ll n_\Gamma$) and $\varepsilon>0$.}
    \Ensure{Two-level preconditioner for the $n_\Gamma \times n_\Gamma$ Schur complement $S_\Gamma$.}
    \State{Draw a random matrix $G\in \mathbb{R}^{n_\Gamma\times (k+p)}$.}
    \State{Compute $F = A_{I\Gamma}G$.}\label{line:compA_IG2}
    \State{Solve $S_I X = F$.}\label{line:solveS_I2}
    \State{Compute $Y = A_{\Gamma I}X$.}\label{line:compA_GI2}\label{line:y=qr}
    \State{Compute $Y = QR$}.
    \State{Set $C = G^\top Y$.}\label{line:GtY}
    \State{Compute the EVD $C = V_1 D_1 V_1^\top + V_2 D_2 V_2^\top$, 
        where $D_1$ contains all the eigenvalues that are at least $\varepsilon$.}\label{line:eigendecomposition}
    \State{Set $T = R V_1 D_1^{-1} V_1^\top R^\top$.}\label{line:T}
    \State{Compute the EVD $T = W E W^\top$.}\label{line:evdT}
    \vspace{0.1cm}
    \State{Set $\widetilde{U} = YW(:, 1 : k)$, $\Sigma = E(1:k, 1:k)$.}\label{line:eigenvector_extension}
    \State{Solve $A_\Gamma Z = \widetilde{U}$.}\label{line:AgU}
    \State{Define the preconditioner $\calM_2 = A_\Gamma^{-1} + Z \Sigma Z^\top$.}\label{line:prec_form}
  \end{algorithmic}
\end{algorithm}
Step 3 is the most expensive step, that is, solving the $n_I \times n_I$ SPD linear system
\begin{equation}
  \label{eq:S_IX=F}
  S_I X = F,
\end{equation}
where $F \in \R^{n_I\times (k+p)}$ and 
$S_I = A_I - A_{I\Gamma} A_{\Gamma}^{-1} A_{\Gamma I}$.
Using an iterative solver requires a linear system solve with $A_\Gamma$ on each iteration.
Importantly for efficiency, the number of iterations can
be limited by employing a large relative tolerance when solving \cref{eq:S_IX=F}
without adversely \cblue{affecting} the performance of the resulting preconditioner.
Numerical experiments in \cref{sec:numerical_experiments} illustrate this robustness.

\cblue{Observe that applying $\calM_2$ to a vector requires 
the solution of a linear system with $A_\Gamma$ and a low rank correction; see Step~\ref{line:prec_form}. }

\subsection{Estimation of the Spectral Condition Number}
\label{sec:spectral_analysis}
In this section, we provide an expectation of the spectral condition 
number of $S_\Gamma$ preconditioned by the Nystr\"om--Schur preconditioner.
Saibaba~\cite{Sai19} derives bounds on the angles between the 
approximate singular vectors computed using a randomized singular
value decomposition and the exact singular vectors of a matrix.
It is straightforward to derive the corresponding bounds for the Nystr\"om method.
Let $\Pi_{M}$ denote the orthogonal projector on the space spanned by the columns of the matrix $M$.
Let $(\lambda_j,u_j)$, $j = 1, \ldots, k$, be the dominant
eigenpairs of $R_\Gamma^{-\top} S_\Gamma R_\Gamma^{-1}$.
Following the notation in~\cref{alg:nystrom}, the angle $\theta_j = \angle(u_j, \widetilde{U})$
between the approximate eigenvectors $\widetilde{U}\in \R^{n_\Gamma \times (k+p)}$ of 
$R_\Gamma^{-\top} S_\Gamma R_\Gamma^{-1}$  and the exact eigenvector $u_j \in \R^{n_\Gamma}$
satisfies
\begin{equation}
  \label{eq:bound_theta_k}
  \sin \angle(u_j, \widetilde{U}) = \|u_j - \Pi_{\widetilde{U}} u_j\|_2 \le \gamma_{j,k}^{q+1} c, 
\end{equation}
where $q$ is the power iteration count (recall \cref{eq:power}), 
$\gamma_{j,k}$ is the gap between $\lambda_j^{-1}-1$ and $\lambda_{k+1}^{-1}-1$ given by
\begin{equation}
\label{eq:gamma}
\gamma_{j,k} = (\lambda^{-1}_{k + 1} - 1)/(\lambda^{-1}_{j} - 1),
\end{equation}
and $c$ has the expected value 
\begin{equation}
  \label{eq:constant}
  \mathbb{E}(c) = \sqrt{\frac{k}{p-1}} + \frac{e\sqrt{(k + p)(n_\Gamma - k)}}{p},
\end{equation}
where $k$ is the required rank and $p\ge2$ is the oversampling parameter.
Hence,
\begin{equation}
  \label{eq:_estimated_bound_theta_k}
  \mathbb{E}\left(\sin \angle(u_j, \widetilde{U})\right) = \mathbb{E}\left(\|u_j - \Pi_{\widetilde{U}} u_j\|_2\right) \le \gamma_{j,k}^{q+1} \mathbb{E}(c).
\end{equation}
Note that if $\lambda_j \le 1/2$ then $\gamma_{j,k} \le 2 \lambda_{j}/\lambda_{k + 1}$ ($j = 1, \ldots, k$).

\begin{proposition}
  \label{prop:expectedkappa}
  Let the EVD of the SPD matrix $I - H = R_\Gamma^{-\top} S_\Gamma R_\Gamma^{-1}$ be 
  $$\begin{bmatrix} U_\perp& U_k \end{bmatrix} \begin{bmatrix} \Lambda_\perp & \\ & \Lambda_k\end{bmatrix} \begin{bmatrix} U_\perp^{\top}\\U_k^{\top} \end{bmatrix},$$ 
    where $\Lambda_\perp \in \R^{(n_\Gamma - k) \times (n_\Gamma - k)}$ and $\Lambda_k \in \R^{k \times k}$ 
    are diagonal matrices with the eigenvalues $(\lambda_{i})_{k\ge i \ge 1}$ and $(\lambda_{i})_{n_\Gamma\ge i \ge k + 1}$, 
    respectively, in decreasing order.
    Furthermore, assume that $\lambda_k \le 1/2$.
  Let the columns of $\widetilde{U}\in\R^{n_\Gamma \times (k+p)}$ be the approximate eigenvectors of $I-H$ computed using 
  the Nystr\"om method and let 
  \[
    \calP = I - (I-H) \widetilde{U} E^{-1} \widetilde{U}^\top \quad \mbox{with}  \quad E = \widetilde{U}^\top (I - H) \widetilde{U},
  \]
  be the associated deflation preconditioner.
    Then, the effective condition number of the two-level preconditioner
    $\calP(I-H) = \calP R_\Gamma^{-\top} S_\Gamma R_\Gamma^{-1}$ satisfies
    \begin{equation}
      \label{eq:expected_condition_number}
      \mathbb{E}\left(\sqrt{\kappa_{\text{\emph{eff}}}\left(\calP (I-H)\right)}\right) \le c_1 \sqrt{\frac{\lambda_{n_\Gamma}}{\lambda_{k + 1}}},
    \end{equation}
    where $c_1^2$ is independent of the spectrum of $I-H$ and can be bounded by a polynomial of degree 3 in $k$.
\end{proposition}
\begin{proof}
  Let $x \in \R^{n_\Gamma}$. Since $u_1, \ldots, u_{n_\Gamma}$ form an orthogonal basis of $\R^{n_\Gamma}$, there 
  exists $\alpha_1, \ldots, \alpha_{n_\Gamma} \in \R$ such that $x = \sum_{i=1}^{n_\Gamma}\alpha_i u_i$.
  In \cite[Theorem 3.4]{KahR17}, Kahl and Rittich show that, if for some positive constant $c_K$,  $\widetilde{U}$ satisfies
  \begin{equation}
    \label{eq:Kahl_ineq}
    \| x - \Pi_{\widetilde{U}} x\|_2^2 \le c_K \frac{\|x\|_{I-H}^2}{{\|I-H\|_2}},
  \end{equation}
  then the effective condition number of $\calP (I-H)$ satisfies 
  $$\kappa_{\text{eff}}\left(\calP (I-H)\right) \le c_K.$$
  Let $t\le k$ and consider
  \begin{align*}
    \| x - \Pi_{\widetilde{U}} x\|_2 &=   \| \sum_{i=1}^{n_\Gamma} \alpha_i u_i - \Pi_{\widetilde{U}} \sum_{i=1}^{n_\Gamma} \alpha_i u_i\|_2\\
                                 &\le \| \sum_{i=t+1}^{n_\Gamma} (I - \Pi_{\widetilde{U}}) \alpha_i u_i \|_2 + \sum_{i = 1}^t |\alpha_i|  \|u_i - \Pi_{\widetilde{U}} u_i\|_2\\
                                 &\le \| \sum_{i=t+1}^{n_\Gamma} \alpha_i u_i \|_2 + \sum_{i = 1}^t |\alpha_i|  \|u_i - \Pi_{\widetilde{U}} u_i\|_2.
  \end{align*}
  The last inequality is obtained using the fact that $I - \Pi_{\widetilde{U}}$ is an orthogonal projector.
  Now bound each term on the right separately. We have
  \begin{align*}
      \| \sum_{i=t+1}^{n_\Gamma} \alpha_i u_i \|_2 
                                 &\le\frac{1}{ \sqrt{\lambda_{t+1}}} \| \sum_{i=t+1}^{n_\Gamma} \sqrt{\lambda_{t+1}}\alpha_i u_i \|_2 
                                 \le\frac{1}{ \sqrt{\lambda_{t+1}}} \| \sum_{i=t+1}^{n_\Gamma} \sqrt{\lambda_{i}}\alpha_i u_i \|_2\\
                                 &\le\frac{1}{ \sqrt{\lambda_{t+1}}}  \sum_{i=t+1}^{n_\Gamma} {\lambda_{i}}\alpha_i^2 
                                 = \frac{1}{\sqrt{\lambda_{t+1}}} \| x - \Pi_{U_t} x \|_{I-H}
                                 = \sqrt{\frac{\lambda_{n_\Gamma}}{\lambda_{t+1}}} \frac{\| x - \Pi_{U_t} x \|_{I-H}}{\sqrt{\|I-H\|_2}}.
  \end{align*}
  From \cref{eq:gamma},  $\gamma_{i,k} \le 1$ for $i = 1, \ldots, t$, thus,
  \begin{align*}
    \sum_{i=1}^t|\alpha_i|\|u_i-\Pi_{\widetilde{U}}u_i\|_2
                                 &\le \sum_{i = 1}^t |\alpha_i| \gamma_{i,k}^{q+1} c
                                 \le c \gamma_{t,k}^{q+\frac{1}{2}} \sum_{i = 1}^t |\alpha_i| \sqrt{\gamma_{i,k}}\\
                                 &= c \gamma_{t,k}^{q+\frac{1}{2}}\sqrt{\lambda_{k+1}^{-1} - 1} \sum_{i = 1}^t |\alpha_i| \frac{1}{\sqrt{\lambda_{i}^{-1} - 1}}\\
                                 &\le c \gamma_{t,k}^{q+\frac{1}{2}} \frac{1}{\sqrt{\lambda_{k+1}}} \sum_{i = 1}^t |\alpha_i| \frac{1}{\sqrt{\lambda_{i}^{-1} - 1}}.
  \end{align*}
  Assuming that $\lambda_{i} \le 1/2$ for $ i = 1, \ldots, t$, we have
  \begin{align*}
    \sum_{i=1}^t|\alpha_i|\|u_i-\Pi_{\widetilde{U}}u_i\|_2 &\le \sqrt{2}c \gamma_{t,k}^{q+\frac{1}{2}}\frac{1}{\sqrt{\lambda_{k+1}}} \sum_{i = 1}^t |\alpha_i| \frac{1}{\sqrt{\lambda_{i}^{-1}}}\\
                                                       &\le \sqrt{2}c \gamma_{t,k}^{q+\frac{1}{2}}\frac{1}{\sqrt{\lambda_{k+1}}} \sum_{i = 1}^t |\alpha_i| \sqrt{\lambda_{i}}.
  \end{align*}
  Using the fact that the $l_1$ and $l_2$ norms are equivalent, we have
  \begin{align*}
    \sum_{i=1}^t|\alpha_i|\|u_i-\Pi_{\widetilde{U}}u_i\|_2 &\le c\sqrt{2t}\gamma_{t,k}^{q+\frac{1}{2}} \frac{1}{\sqrt{\lambda_{k+1}}} \sqrt{\sum_{i = 1}^t \alpha_i^2 \lambda_i} \\
                                 &= c \sqrt{2t}\gamma_{t,k}^{q+\frac{1}{2}} \frac{1}{\sqrt{\lambda_{k+1}}} \| \Pi_{U_t} x \|_{I-H}\\
                                 &= c\sqrt{2t}\gamma_{t,k}^{q+\frac{1}{2}} \sqrt{\frac{\lambda_{n_\Gamma}}{\lambda_{k+1}}} \frac{\| \Pi_{U_t} x \|_{I-H}}{\sqrt{\|I-H\|_2}}.
  \end{align*}
  Since $\lambda_k\ge \lambda_t$ we have 
  \begin{align*}
     \sum_{i=1}^t|\alpha_i|\|u_i-\Pi_{\widetilde{U}}u_i\|_2 &\le c\sqrt{2t}\gamma_{t,k}^{q+\frac{1}{2}} \sqrt{\frac{\lambda_{n_\Gamma}}{\lambda_{t+1}}} \frac{\| \Pi_{U_t} x \|_{I-H}}{\sqrt{\|I-H\|_2}}.
  \end{align*}
It follows that
\begin{align*}
    \| x - \Pi_{\widetilde{U}} x\|_2 &\le \sqrt{\frac{\lambda_{n_\Gamma}}{\lambda_{t+1}}} \frac{\| x - \Pi_{U_t} x \|_{I-H}}{\sqrt{\|I-H\|_2}} + c \sqrt{2t} \gamma_{t,k}^{q+\frac{1}{2}} \sqrt{\frac{\lambda_{n_\Gamma}}{\lambda_{t+1}}} \frac{\| \Pi_{U_t} x \|_{I-H}}{\sqrt{\|I-H\|_2}}\\  
                                 &\le \sqrt{2} \max(c\sqrt{2t}\gamma_{t,k}^{q+\frac{1}{2}} , 1) \sqrt{\frac{\lambda_{n_\Gamma}}{\lambda_{t+1}}} \frac{\|x\|_{I-H}}{\sqrt{\|I-H\|_2}}.
  \end{align*}
  Hence \cref{eq:Kahl_ineq} is satisfied and we have
  \[
    \kappa_{\text{eff}}\left(\calP (I-H)\right) \le 2 \max(2c^2 t\gamma_{t,k}^{2q+1}, 1) \frac{\lambda_{n_\Gamma}}{\lambda_{t+1}}.
  \]
Thus,
 \[
    \mathbb{E}\left(\sqrt{\kappa_{\text{eff}}\left(\calP (I-H)\right)} \right) \le \sqrt{2} \max( \mathbb{E}(c) \sqrt{2t}\gamma_{t,k}^{q+\frac{1}{2}} , 1) \sqrt{\frac{\lambda_{n_\Gamma}}{\lambda_{t+1}}}.
  \]
  Since $t$ is chosen arbitrarily between $1$ and $k$ we have
 \begin{equation}
   \label{eq:tight_bound}
   \mathbb{E}\left(\sqrt{\kappa_{\text{eff}}\left(\calP (I-H)\right)} \right) \le \sqrt{2} \min_{1\le t \le k}\left(\max\left( \mathbb{E}(c) \sqrt{2t}\gamma_{t,k}^{q+\frac{1}{2}} , 1\right)\sqrt{\frac{\lambda_{n_\Gamma}}{\lambda_{t+1}}}\right).
 \end{equation}

  Because $\mathbb{E}(c)$ can be bounded by a polynomial of degree 1 in $k$ and $\gamma_{t,k}\le 1$, $\max(4 t \gamma_{t,k}^{2q+1} \left(\mathbb{E}(c)\right)^2, 2)$ 
  can be bounded by a polynomial of degree 3 in $k$ independent of the spectrum of $I-H$.
\end{proof}
Note that, in practice, when the problem is challenging, a few eigenvalues 
of $R_\Gamma^{-\top} S_\Gamma R_\Gamma^{-1}$ are close to the origin. This is reflected in a rapid 
and exponential decay of the values of the entries of $\Lambda^{-1} - I$.
\Cref{fig:estkappa} depicts the bound obtained in~\cref{prop:expectedkappa} for different values of $k$ and $q$ for problem s3rmt3m3.

\begin{figure}[htbp]
  \label{fig:estkappa}
  \centering
    \includegraphics[height=4.5cm]{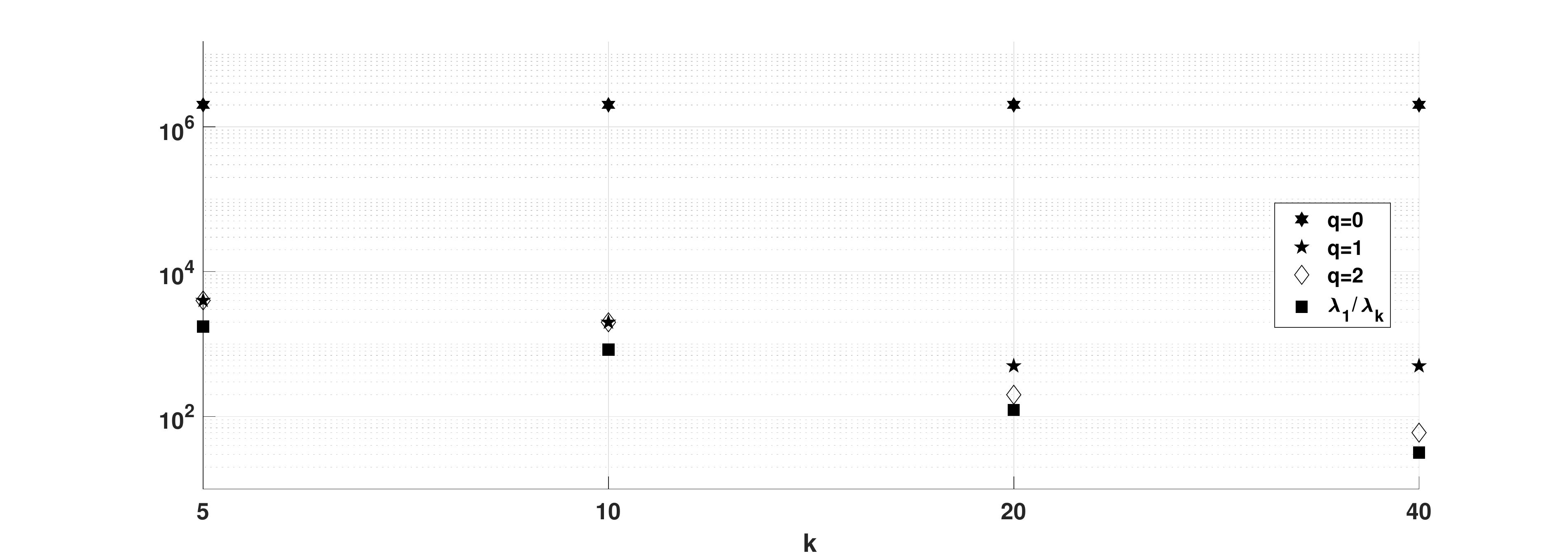}
  \caption{Problem s3rmt3m3: Values of the bound \cref{eq:tight_bound} on
  $\left(\mathbb{E}\left(\sqrt{\kappa_{\text{eff}}\left(\calP (I-H)\right)} \right)\right)^2$  for a range of values of $k$ and $q$. }
\end{figure}

\section{Numerical Experiments}
\label{sec:numerical_experiments}
We use 64 subdomains (i.e., $A_I$ is a 64-block diagonal matrix) for each of our test matrices with the exception of one
problem.  The matrix nd3k is much denser than the others, and we use only two blocks (to reduce the runtime).
For comparison purposes, we include results for the Schur complement
preconditioners $\widetilde{S}_1$ and $\widetilde{S}_2$ given by \cref{eq:S1} and \cref{eq:S2}, respectively.
As demonstrated in~\cref{sec:nonefficientLanczos}, the latter is too costly to be practical, 
however, its performance is the ideal since it guarantees the smallest spectral condition number 
for a fixed deflation subspace.
Therefore, the quality of the Nystr\"om--Schur preconditioner will be measured in terms of how close 
its performance is to that of $\widetilde{S}_2$ and the reduction in iteration it gives compared to $\widetilde{S}_1$.
For a given problem, the right-hand side vector is the same for all the
tests: it is generated randomly
with entries from the standard normal distribution. 
The relative convergence tolerance for PCG is $10^{-6}$.
Unless otherwise specified, the parameters within Nystr\"om's method (\cref{alg:nystrom}) are rank $k=20$, 
oversampling $p=0$, \cblue{and power iteration $q=0$.}
To ensure fair comparisons, the random matrices generated in different runs of the Nystr\"om algorithm use the same seed.
We employ the Nystr\"om--Schur variant $\calM_2$ \cref{eq:left_prec_mat} (recall that its construction does not require the 
Cholesky factors of $A_\Gamma$).
The relative convergence tolerance used when solving the SPD system \cref{eq:S_IX=F} is $\varepsilon_{S_I}= 0.1$.
This system \cref{eq:S_IX=F} is preconditioned by the block diagonal matrix $A_I$.
We denote by $it_{S_I}$  the number of block PCG iterations required to solve \cref{eq:S_IX=F}
during the construction of the   Nystr\"om--Schur preconditioners (it is zero for $\widetilde{S}_1$ and $\widetilde{S}_2$),
and by $it_{\text{PCG}}$ the  PCG iteration count for solving \cref{eq:Sx=b}.
The     total number of iterations is $it_{\text{total}} = it_{S_I} + it_{\text{PCG}}$.
We use the code \cite{AldRS21} to generate the numerical experiments.

\subsection{Linear system with $S_I$}
We start by considering how to efficiently compute an approximate  
solution of \cref{eq:S_IX=F}.

\subsubsection{Block and classic CG}
The system  \cref{eq:S_IX=F} has $k+p$ right hand sides.
The number of iterations required by PCG to solve each right hand side is different
and the variation can be large; this
is illustrated in \Cref{fig:histogram1rhs} for problem bcsstk38. \cblue{Here we report the
number of right hand sides for which the iteration count lies  in the interval $[k,k+10)$,
$k = 100, \ldots, 240$. For example, there are 4 right hand sides for which the count
is between 110 and 119.}
Similar behaviour was observed for our other test problems.
\begin{figure}
  \label{fig:histogram1rhs}
  \centering
    \includegraphics[height=4.5cm]{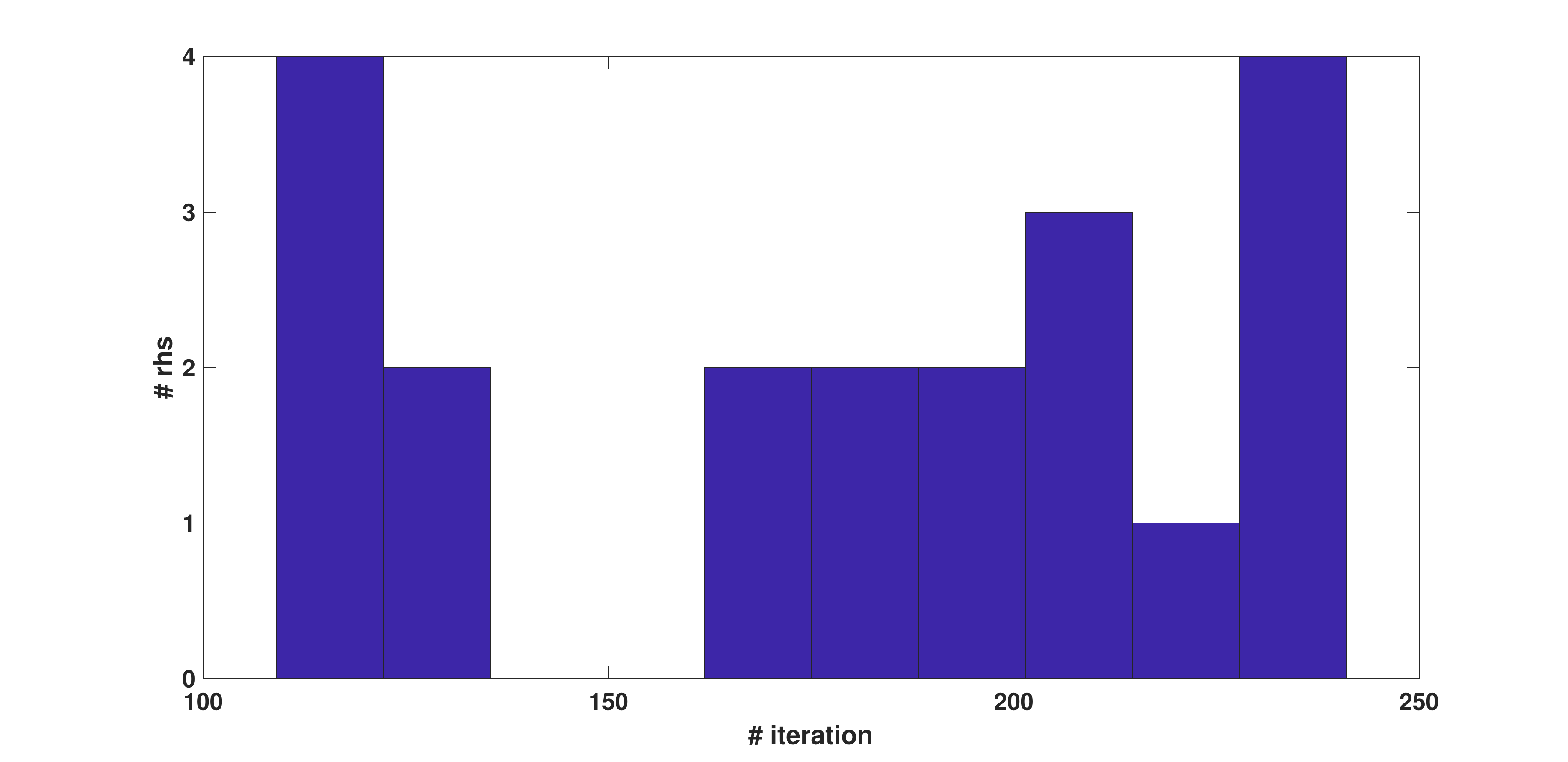}
  \caption{Histogram of the PCG iteration counts for \cref{eq:S_IX=F} for problem bcsstk38.
 \cblue{The number of right hand sides for which the iteration count is between $[k,k+10)$,
$k = 100, \ldots, 240$, is given.} }
\end{figure}

\Cref{tab:blockvsclassic}  reports the iteration counts for the classical PCG method
and the breakdown-free block PCG  method \cite{JiL17,Ole80}.
For PCG, $iters$ is the largest PCG
iteration count over the $k+p$ right hand sides. For the block method, $iters=it_{S_I}$ is
the number of block PCG iterations. As expected from the theory, the block method significantly 
reduces the (maximum) iteration count.
For our examples, it also leads to a modest reduction in the iteration count $it_{\text{PCG}}$
for solving \cref{eq:Sx=b}.
\begin{table}
  \label{tab:blockvsclassic}
  \centering
  \footnotesize{
  \begin{tabular}{lrrcrr}
    \toprule
              &  \multicolumn{2}{c}{Classic} && \multicolumn{2}{c}{Block}          \\
    \cmidrule(lr){2-3} \cmidrule(lr){5-6}
    Identifier   & $iters$ &  $it_{\text{PCG}}$ && $iters$ &  $it_{\text{PCG}}$  \\
    \midrule
    bcsstk38  & 238  & 186                   && 46         & 173 \\  
    el2d      & 549  & 261                   && 72         & 228 \\
    el3d      & 95   & 56                    && 24         & 52  \\ 
    msc10848  & 203  & 194                   && 47         & 166 \\
    nd3k      & 294  & 191                   && 32         & 178 \\
    s3rmt3m3  & 403  & 157                   && 37         & 98  \\
    \bottomrule
  \end{tabular}
  }
    \vspace{0.05in}
  \caption{A comparison of the performance of classic and block PCG. $iters$ denotes the iteration count
  for solving \cref{eq:S_IX=F} (details in the text) and $it_{\text{PCG}}$ is the iteration count for solving \cref{eq:Sx=b}.}
\end{table}

\subsubsection{Impact of tolerance $\varepsilon_{S_I}$}
We now study the impact of the convergence tolerance $\varepsilon_{S_I}$ 
used when solving \cref{eq:S_IX=F} on the quality of the Nystr\"om--Schur preconditioner.
In \Cref{tab:compare_tolerance}, we present results
for three test problems that illustrate the (slightly) different behaviors we observed.
The results demonstrate numerically that a large tolerance can be used
without affecting the quality of the preconditioner.
Indeed, using $\varepsilon_{S_I}= 0.3$ leads
to a preconditioner whose efficiency is  close to that of the ideal 
(but impractical) two-level preconditioner $\widetilde{S}_2$.
The use of a large $\varepsilon_{S_I}$ to limit
$it_{S_I}$ is crucial in ensuring low construction costs for the Nystr\"om--Schur preconditioners.

\begin{table}
  \label{tab:compare_tolerance}
  \centering
  \footnotesize{
  \begin{tabular}{lrrrcrcr}
    \toprule
                               &               &  \multicolumn{2}{c}{$\calM_2$} &&     ${\widetilde{S}_1}$    && ${\widetilde{S}_2}$    \\
    \cmidrule(lr){3-4}
    \multicolumn{1}{c}{Identifier}& \multicolumn{1}{c}{$\varepsilon_{S_I}$} &  $it_{S_I}$ & $it_{PCG}$  &&     &&     \\
    \midrule
    \multirow{5}{*}{el2d    }  & 0.8\phantom{0}&  1          & 500+             && \multirow{5}{*}{914}  && \multirow{5}{*}{231} \\
                               & 0.5\phantom{0}&  68         & 228              &&                       &&                       \\
                               & 0.3\phantom{0}&  70         & 228              &&                       &&                       \\
                               & 0.1\phantom{0}&  72         & 228              &&                       &&                       \\
                               & 0.01          &  78         & 228              &&                       &&                       \\
                               \midrule
    \multirow{5}{*}{el3d    }  & 0.8\phantom{0}&  1          & 173              && \multirow{5}{*}{174}  && \multirow{5}{*}{37}   \\
                               & 0.5\phantom{0}&  2          & 171              &&                       &&                       \\
                               & 0.3\phantom{0}&  22         & 52               &&                       &&                       \\
                               & 0.1\phantom{0}&  24         & 52               &&                       &&                       \\
                               & 0.01          &  27         & 52               &&                       &&                       \\
                               \midrule
    \multirow{5}{*}{nd3k    }  & 0.8\phantom{0}&  32         & 178              && \multirow{5}{*}{603 } && \multirow{5}{*}{143}  \\
                               & 0.5\phantom{0}&  32         & 178              &&                       &&                       \\
                               & 0.3\phantom{0}&  32         & 178              &&                       &&                       \\
                               & 0.1\phantom{0}&  32         & 178              &&                       &&                       \\
                               & 0.01          &  33         & 178              &&                       &&                       \\
    \bottomrule
    
  \end{tabular}  }
  \vspace{0.05in}
  \caption{The effects of the convergence tolerance $\varepsilon_{S_I}$ on the quality of the Nystr\"om--Schur preconditioner.}
\end{table}

\subsection{Type of preconditioner}
We next compare the performances of
the variants $\calM_i$ and $\calMADi{i-}$ ($i = 1,2,3$)  of the Nystr\"om--Schur preconditioner presented in~\cref{sec:preconditioner}.
In~\cref{tab:compare_variations}, we report the total iteration count $it_\text{total}$.
All the variants have similar behaviors and have a significantly smaller count than the one-level preconditioner $\widetilde{S}_1$.

\begin{table}
  \label{tab:compare_variations}
  \centering
  \footnotesize{
  \begin{tabular}{lrrrrrrrr}
    \toprule
    Identifier   & $\calM_1$ & $\calMADi{1-}$ & $\calM_2$ & $\calMADi{2-}$ & $\calM_3$ & $\calMADi{3-}$ & $\widetilde{S}_1$ & $\widetilde{S}_2$\\
    \midrule                                                                                                                                 
    bcsstk38 & 218       & 218            & 219       & 219            & 360       & 313            &  584          &   122        \\
    el2d     & 266       & 267            & 300       & 300            & 282       & 282            &  914          &   231        \\
    el3d     & 73        & 72             & 76        & 75             & 78        & 76             &  174          &   37         \\
    msc10848 & 206       & 205            & 213       & 211            & 216       & 222            &  612          &   116        \\
    nd3k     & 205       & 205            & 210       & 210            & 211       & 211            &  603          &   143        \\
    s3rmt3m3 & 127       & 127            & 135       & 134            & 161       & 153            &  441          &   70         \\
    \bottomrule
  \end{tabular}
  }
      \vspace{0.05in}
    \caption{Comparison of $it_\text{total}$ for the variants of the Nystr\"om--Schur preconditioner
    and $\widetilde{S}_1$ and $\widetilde{S}_2$. $\varepsilon_{S_I}=0.1$.}
\end{table}

\subsection{Varying the rank and the oversampling parameter}
We now look  at varying the rank $k$ within the Nystr\"om algorithm and
demonstrate numerically that the efficiency of the preconditioner is robust with respect to the oversampling parameter $p$.
For problem s3rmt3m3, \Cref{tab:rank_variation} compares the  iteration counts for $\calM_2$ with that of
the ideal two-level preconditioner $\widetilde{S}_2$ for $k$ ranging from 5 to 320.
For $\widetilde{S}_1$, the iteration count is 441.
This demonstrates the effectiveness of the Nystr\"om--Schur preconditioner in reducing the iteration count.
Increasing the size of the deflation subspace (the rank $k$) 
steadily reduces the iteration count required to solve the $S_I$ system  \cref{eq:S_IX=F}.
For the same test example, \Cref{tab:oversampling_variation} presents the iteration counts for a range of values of
the oversampling parameter $p$ (here $k=20$).
We observe that the counts are relatively insensitive to $p$ but, as $p$ increases, $it_{\text{PCG}}$
reduces towards the lower bound of 70 PCG iterations required by $\widetilde{S}_2$.
Similar behavior was noticed for our other test examples.
Although increasing $k$ and $p$ improves the efficiency of the Nystr\"om--Schur preconditioner, 
this comes with  extra costs during both the construction of the preconditioner and its application.
Nevertheless, the savings from the reduction in the iteration count and the efficiency in solving 
block linear systems of equations for moderate block sizes (for example, $k=40$) typically outweigh the 
increase in construction costs.
\begin{table}
  \label{tab:rank_variation}
  \centering
  \footnotesize{
  \begin{tabular}{llrrrrrrr}
    \toprule
                               & $k$                 & 5   & 10  & 20  & 40  & 80  & 160 & 320 \\
                               \midrule
    \multirow{2}{*}{$\calM_2$} & $it_{S_I}$          & 97  & 57  & 37  & 23  & 16  & 11  & 8   \\ 
                               & $it_{\text{PCG}}$   & 244 & 203 & 98  & 53  & 30  & 20  & 14  \\
                               \midrule
                 $\widetilde{S}_2$ & $it_{\text{PCG}}$   & 212 & 153 & 70  & 37  & 22  & 13  & 9   \\
    \bottomrule
  \end{tabular}
  }
      \vspace{0.05in}

  \caption{Problem s3rmt3m3: Impact of the rank $k$ on the iteration counts ($p=0$).}
\end{table}

\begin{table}
  \label{tab:oversampling_variation}
  \centering
  \footnotesize{
  \begin{tabular}{lrrrrr}
  \toprule
                                $p$                 & 0   & 5   & 10 & 20 & 40\\
                               \midrule
                                $it_{S_I}$          & 37  & 31  & 28 & 23 & 20\\
                                 $it_{\text{PCG}}$ & 98  & 86  & 79 & 77 & 74\\
    \bottomrule
  \end{tabular}
  }
        \vspace{0.05in}

  \caption{Problem s3rmt3m3: Impact of the oversampling parameter $p$ on the iteration counts ($k=20$).}
\end{table}

\subsection{Comparisons with incomplete Cholesky factorization preconditioners}
Finally, we compare the Nystr\"om--Schur preconditioner with two incomplete Cholesky factorization preconditioners 
applied to original system.
The first is the Matlab variant \verb|ichol| with the global diagonal shift set
to $0.1$ \cblue{and default values for other parameters} and the second is the Matlab interface to the incomplete Cholesky (IC) factorization preconditioner
\verb|HSL_MI28| \cite{ScoT14} from the HSL library \cite{hsl:2018}  using the default parameter settings.
IC preconditioners are widely used but  their construction is often serial, potentially limiting their 
suitability for very large problems (see \cite{hsth:2018} for an IC preconditioner that
can be parallelised).
\begin{table}
  \label{tab:ichol}
  \centering
    \footnotesize{
  \begin{tabular}{lrrrr}
    \toprule      
    Identifier & \multicolumn{2}{c}{$\mathcal{M}_2$} & \verb|HSL_MI28| & \verb|ichol| \\
    \cmidrule(lr){2-3}
               & $it_{S_I}$   & $it_{PCG}$\\
    \midrule      
    bcsstk38   &46 &173              & 593             & 2786         \\  
    ela2d      &72 &228              & 108             & 2319         \\  
    ela3d      &24 &52               & 36              & 170          \\  
    msc10848   &47 &166              & 145             & 784          \\  
    nd3k       &32 &178              & 102             & 1231         \\  
    s3rmt3m3   &37 &98               & 610             & 2281         \\  
    \bottomrule      
  \end{tabular}
  }
  \cprotect\caption{PCG iteration counts for the Nystr\"om--Schur preconditioner $\mathcal{M}_2$ (with $k=20$)
  and the IC preconditioners \verb|HSL_MI28| and \verb|ichol|.}
\end{table}
\cblue{In terms of iteration counts,} the Nystr\"om--Schur and the \verb|HSL_MI28| preconditioners are clearly superior to 
the simple \verb|ichol|
preconditioner,
with neither consistently offering the best performance.
\Cref{fig:ichol} presents the residual norm history for PCG. This is confirmed
by the results in the Appendix for our large test set. The residual norm
for $\mathcal{M}_2$ decreases monotonically while for the
IC preconditioners we observe oscillatory behaviour.
\begin{figure}
\begin{center}
  \label{fig:ichol}
  \centering
    \includegraphics[height=4.5cm]{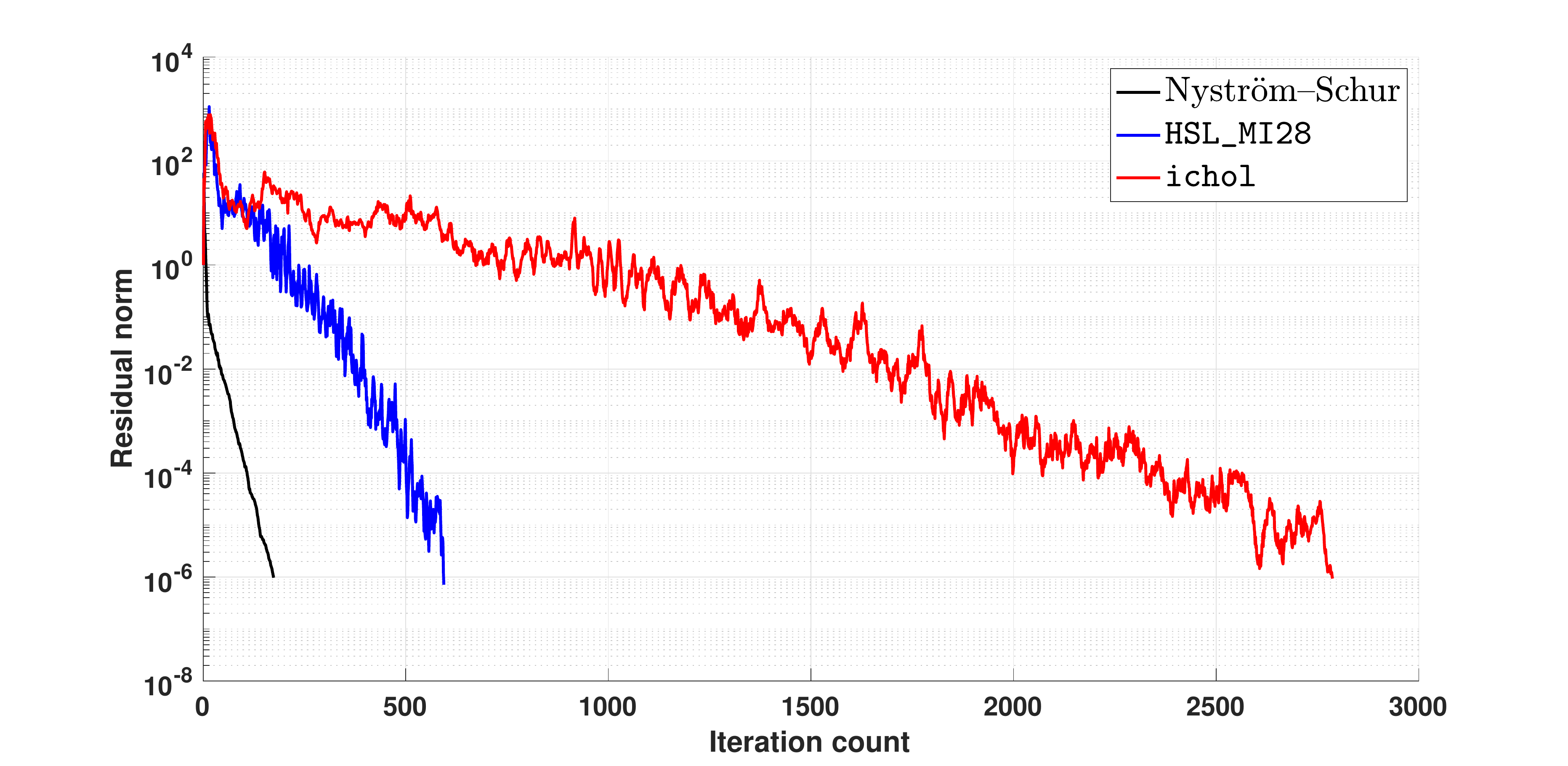}
    \includegraphics[height=4.5cm]{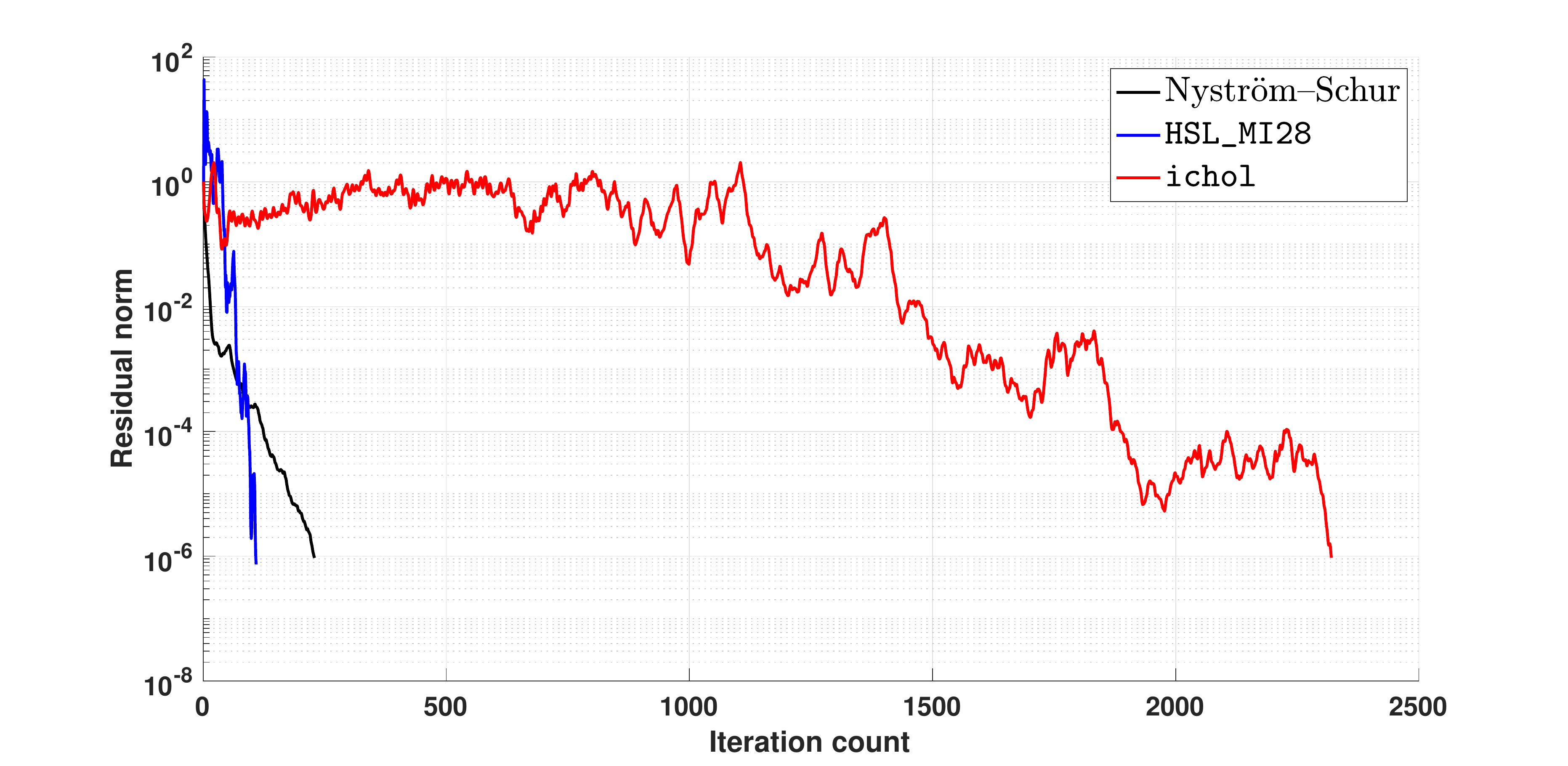}
  \caption{PCG residual norm history for test examples bcsstk38 (top) and ela2d (bottom).}
  \end{center}
\end{figure}

\cblue{Because our implementation of the Nystr\"om--Schur preconditioner is in Matlab, we are not able to
provide performance comparisons in terms of computation times. 
Having demonstrated the potential of our two-level Nystr\"om--Schur preconditioner, 
one of our objectives for the future is to
develop an efficient (parallel) implementation in Fortran that will be included within the HSL library.
This will allow users to test out the preconditioner and to assess the performance of 
both constructing and applying the preconditioner. Our preliminary work on this is
encouraging.}

\section{Concluding comments}
\label{sec:conclusion}
In this paper, we have investigated  using randomized methods to
construct efficient and robust preconditioners for use with CG to solve large-scale
SPD linear systems. The approach requires an initial ordering to
doubly bordered block diagonal form and then uses
a Schur complement approximation.
We have demonstrated that by carefully posing the approximation problem we
can apply randomized methods to construct high quality
preconditioners, which gives an improvement over previously proposed
methods that use low rank approximation strategies.
 We have presented a number of 
variants of our new Nystr\"om--Schur preconditioner. During the preconditioner construction,
we must solve a smaller linear system with multiple right-hand sides.
Our numerical experiments have shown that a small number of
iterations of block CG are needed to obtain an approximate solution that is sufficient 
to construct an effective preconditioner.

Currently, the construction and application of our Nystr\"om--Schur preconditioners requires the solution 
of linear systems with the block matrix $A_\Gamma$ \cref{eq:perm_A}.
Given the promising results presented in this paper, in the future
we plan to investigate employing a recursive approach, following ideas given in~\cite{XiLS16}.
This will only require the solution of systems involving a much smaller matrix and will lead
to a practical approach for very large-scale SPD systems.
\cblue{A parallel implementation of the preconditioner will also be developed.}

\appendix
\section{Extended numerical experiments}
\cblue{Here we present results for a larger test set. The problems are given in Table~\ref{tab:SSmatrices}.
We selected all the SPD matrices in the SuiteSparse Collection with $n$ lying between 5K and 100K,
giving us a set of 71 problems. For each problem, we ran PCG with the $\widetilde{S}_1$,
$\calM_2$, ${\widetilde{S}_2}$ and ${\tt HSL\_MI28}$ preconditioners. 
In all the tests, we use 64 subdomains. For $\calM_2$, we used $k=20$ and set $p=q=0$.
Iteration counts are given in  the table, whilst performance profiles \cite{dojo:2002} are presented in Figure~\ref{fig:performance_profile}.
In recent years, performance profiles have become a popular and widely used tool for providing objective information
when benchmarking algorithms. The performance profile takes into account the number of problems solved by an algorithm as well 
as the cost to solve it. It scales the cost of solving the problem according to the best solver for that problem.
In our case, the performance cost is the iteration count (for $\calM_2$,
we sum the counts $it_{S_I}$ and $it_{PCG}$).
Note that we do not include ${\widetilde{S}_2}$ in the performance profiles because it is an ideal
but impractical two-level  preconditioner and, as such, it always outperforms $\calM_2$. 
The performance profile shows that on the problems where $\widetilde{S}_1$ struggles,
there is little to choose between the overall quality of  $\calM_2$ and ${\tt HSL\_MI28}$.}
\begin{table}[htbp]
  \label{tab:SSmatrices}
  \centering
  {
  \footnotesize
  \begin{adjustbox}{width=0.45\linewidth}
  \begin{tabular}{lrrrrrr}
    \toprule
                                &  $\widetilde{S}_1$ &    \multicolumn{2}{c}{$\calM_2$} & ${\widetilde{S}_2}$ &  ${\tt HSL\_MI28}$  & $\kappa(A)$\\
    \cmidrule(lr){3-4}
    \multicolumn{1}{c}{Identifier} &        & $it_{S_I}$ & $it_{PCG}$  &        &        &    \\
    \midrule
    aft01                       & 118    &  19   &  45     & 31     & 17      & 9e+18  \\
    apache1                     & 667    &  122  &  291    & 192    & 72      & 3e+06  \\
    bcsstk17                    & 349    &  46   &  55     & 48     & 59      & 1e+10  \\
    bcsstk18                    & 136    &  40   &  77     & 45     & 26      & 6e+11  \\
    bcsstk25                    & $\dag$ &  92   &  660    & 453    & 254     & 1e+13  \\
    bcsstk36                    & 451    &  64   &  214    & 169    & $\dag$  & 1e+12  \\
    bcsstk38                    & 584    &  46   &  171    & 122    & 593     & 6e+16  \\
    bodyy6                      & 182    &  53   &  163    & 129    & 5       & 9e+04  \\
    cant                        & $\dag$ &  57   &  228    & 396    & 933     & 5e+10  \\
    cfd1                        & 209    &  30   &  72     & 50     & 274     & 1e+06  \\
    consph                      & 185    &  47   &  177    & 136    & 50      & 3e+07  \\
    gridgena                    & 426    &  90   &  377    & 298    & 66      & 6e+05  \\
    gyro                        & $\dag$ &  55   &  346    & 518    & 319     & 4e+09  \\
    gyro\_k                     & $\dag$ &  55   &  346    & 518    & 319     & 3e+09  \\
    gyro\_m                     & 165    &  16   &  34     & 22     & 17      & 1e+07  \\
    m\_t1                       & 867    &  85   &  247    & 187    & $\ddag$ & 3e+11  \\
    minsurfo                    & 15     &  3    &  15     & 13     & 3       & 8e+01  \\
    msc10848                    & 612    &  47   &  168    & 116    & 145     & 3e+10  \\
    msc23052                    & 479    &  69   &  220    & 175    & $\ddag$ & 1e+12  \\
    nasasrb                     & 1279   &  135  &  496    & 421    & $\dag$  & 1e+09  \\
    nd3k                        & 1091   &  56   &  301    & 230    & 102     & 5e+07  \\
    nd6k                        & 1184   &  108  &  325    & 248    & 116     & 6e+07  \\
    oilpan                      & 647    &  67   &  122    & 72     & 507     & 4e+09  \\
    olafu                       & 1428   &  69   &  489    & 757    & 557     & 2e+12  \\
    pdb1HYS                     & 869    &  89   &  83     & 274    & 483     & 2e+12  \\
    vanbody                     & $\dag$ &  287  &  1106   & 769    & $\ddag$ & 4e+03  \\
    ct20stif                    & 1296   &  90   &  232    & 281    & $\dag$  & 2e+14  \\
    nd12k                       & 1039   &  155  &  337    & 265    & 111     & 2e+08  \\
    nd24k                       & 1093   &  165  &  386    & 268    & 120     & 2e+08  \\
    s1rmq4m1                    & 154    &  19   &  50     & 32     & 33      & 5e+06  \\
    s1rmt3m1                    & 192    &  24   &  59     & 39     & 18      & 3e+08  \\
    s2rmq4m1                    & 231    &  28   &  54     & 41     & 39      & 4e+08  \\
    s2rmt3m1                    & 260    &  31   &  64     & 45     & 33      & 3e+11  \\
    s3dkq4m2                    & $\dag$ &  148  &  339    & 236    & 610     & 6e+11  \\
    \bottomrule
  \end{tabular} 
  \end{adjustbox}
    \quad \quad 
  \begin{adjustbox}{width=0.45\linewidth}
  \begin{tabular}{lrrrrrr}
    \toprule
                                &  $\widetilde{S}_1$ &    \multicolumn{2}{c}{$\calM_2$} & ${\widetilde{S}_2}$  &  ${\tt HSL\_MI28}$  & $\kappa(A)$\\
    \cmidrule(lr){3-4}
    \multicolumn{1}{c}{Identifier} &         & $it_{S_I}$ & $it_{PCG}$ &        &                 & \\
    \midrule
    s3dkt3m2                    & $\dag$ &  164  &  338    & 270    &  1107  & 3e+10  \\
    s3rmq4m1                    & 356    &  31   &  80     & 58     &  472   & 4e+10  \\
    s3rmt3m1                    & 434    &  36   &  101    & 64     &  413   & 4e+10  \\
    s3rmt3m3                    & 441    &  37   &  101    & 70     &  610   & 3e+00  \\
    ship\_001                   & 1453   &  367  &  600    & 368    &  1177   & 6e+09  \\
    smt                         & 399    &  59   &  112    & 72     &  95     & 1e+09  \\
    thermal1                    & 169    &  30   &  62     & 47     &  30     & 4e+01  \\
    Pres\_Poisson               & 92     &  13   &  29     & 19     &  32     & 3e+06  \\
    crankseg\_1                 & 92     &  16   &  49     & 33     &  34     & 9e+18  \\
    crankseg\_2                 & 89     &  17   &  47     & 32     &  38     & 8e+06  \\
    Kuu                         & 81     &  16   &  44     & 31     &  10     & 3e+04  \\
    bodyy5                      & 72     &  19   &  67     & 57     &  4      & 9e+03  \\
    Dubcova2                    & 62     &  11   &  32     & 23     &  14     & 1e+04  \\
    cbuckle                     & 55     &  9    &  51     & 39     &  47     & 7e+07  \\
    fv3                         & 50     &  12   &  31     & 21     &  8      & 4e+03  \\
    Dubcova1                    & 39     &  8    &  24     & 15     &  7      & 2e+03  \\
    bodyy4                      & 34     &  8    &  29     & 24     &  4      & 1e+03  \\
    jnlbrng1                    & 22     &  4    &  21     & 19     &  4      & 1e+02  \\
    bundle1                     & 13     &  3    &  8      & 5      &  5      & 1e+04  \\
    t2dah\_e                    & 12     &  3    &  12     & 11     &  3      & 3e+07  \\
    obstclae                    & 12     &  3    &  12     & 12     &  3      & 4e+01  \\
    torsion1                    & 12     &  3    &  12     & 12     &  3      & 8e+03  \\
    wathen100                   & 12     &  3    &  12     & 11     &  3      & 2e+07  \\
    wathen120                   & 12     &  3    &  12     & 11     &  3      & 2e+07  \\
    fv1                         & 7      &  2    &  7      & 7      &  3      & 1e+01  \\
    fv2                         & 7      &  2    &  7      & 7      &  3      & 1e+01  \\
    shallow\_water2             & 7      &  40   &  7      & 7      &  3      & 3e+12  \\
    shallow\_water1             & 5      &  20   &  5      & 5      &  2      & 1e+01  \\
    Muu                         & 6      &  1    &  6      & 6      &  2      & 1e+02  \\
    qa8fm                       & 6      &  1    &  6      & 6      &  2      & 1e+02  \\
    crystm02                    & 6      &  1    &  6      & 5      &  2      & 4e+02  \\
    crystm03                    & 6      &  1    &  6      & 5      &  2      & 4e+02  \\
    finan512                    & 5      &  1    &  5      & 5      &  3      & 9e+01  \\
    ted\_B\_unscaled            & 3      &  1    &  3      & 4      &  2      & 4e+05  \\
    ted\_B                      & 2      &  1    &  3      & 3      &  2      & 2e+11  \\
    Trefethen\_20000b           & 3      &  1    &  2      & 2      &  3      & 1e+05  \\
    Trefethen\_20000            & 4      &  1    &  2      & 2      &  3      & 2e+05  \\  
                                &        &       &         &        &         &        \\  
    \bottomrule
  \end{tabular}
  \end{adjustbox}
  }
  \vspace{0.1in}
  \caption{PCG iteration counts for SPD matrices from the SuiteSparse Collection with $n$ ranging between 5K and 100K.
  }
\end{table}

\begin{figure}
\begin{center}
  \label{fig:performance_profile}
  \centering
    \includegraphics[height=4.5cm]{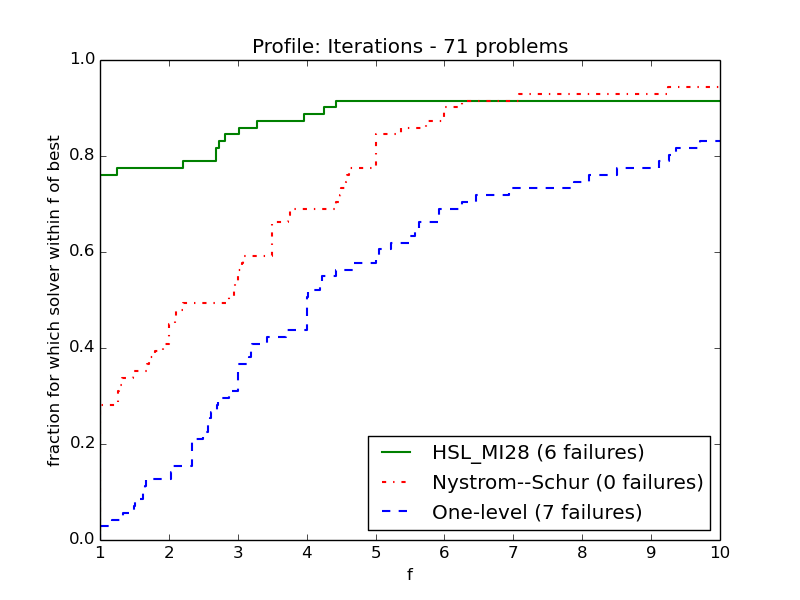}
    \includegraphics[height=4.5cm]{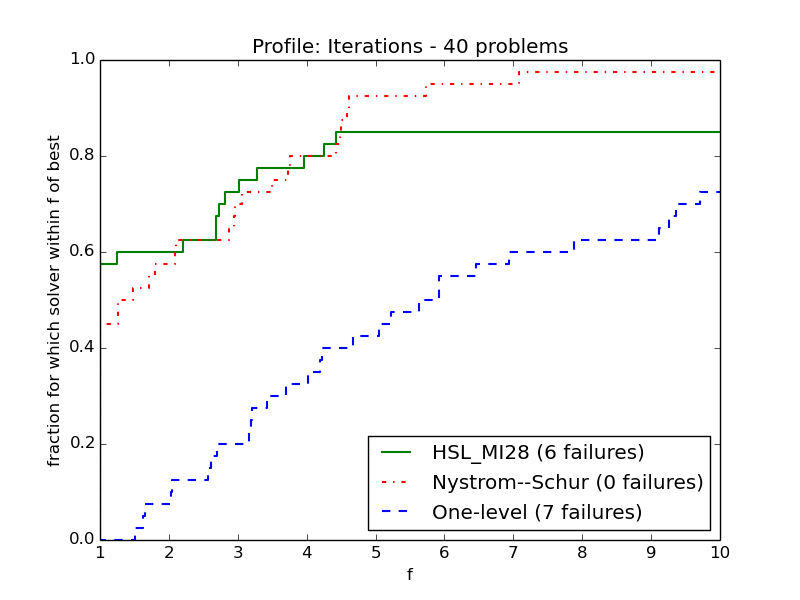}
  \caption{Iteration count performance profile for the large test set. The 40 problems used in the right hand plot are  the
  subset for which the $\widetilde{S}_1$ (one-level) iteration count exceeded 100.}
  \end{center}
\end{figure}

\bibliographystyle{siamplain}
\bibliography{main}
\end{document}